\documentclass{amsart}

\usepackage{amssymb, amsmath, amsthm}

\usepackage{mathrsfs}
\usepackage{amscd}
\usepackage{verbatim}
\usepackage{a4wide}

\usepackage{enumerate}

\allowdisplaybreaks

\usepackage{hyperref}

\theoremstyle{plain}
\newtheorem{theorem}{Theorem}[section]
\theoremstyle{remark}
\newtheorem{remark}[theorem]{Remark}
\newtheorem{example}[theorem]{Example}

\theoremstyle{plain}
\newtheorem{corollary}[theorem]{Corollary}
\newtheorem{lemma}[theorem]{Lemma}
\newtheorem{proposition}[theorem]{Proposition}
\newtheorem{definition}[theorem]{Definition}

\numberwithin{equation}{section}

\def\N{{\mathbb N}}

\def\R{{\mathbb R}}

\def\C{{\mathbb C}}

\newcommand{\F}{{\mathscr F}}

\newcommand{\one}{{{\bf 1}}}

\newcommand{\wh}{\widehat}
\newcommand{\supp}{\text{\rm supp\,}}

\newcommand{\Schw}{{\mathscr S}}
\newcommand{\TD}{{\mathscr S'}}

\def\typeout#1{\message{^^J}\message{#1}\message{^^J}}
\newif\ifSRCOK \SRCOKtrue
\newcount\PAGETOP
\newcount\LASTLINE
\global\PAGETOP=1
\global\LASTLINE=-1
\def\EJECT{\SRC\eject}
\def\WinEdt#1{\typeout{:#1}}
\gdef\MainFile{\jobname.tex}
\gdef\CurrentInput{\MainFile}
\newcount\INPSP
\global\INPSP=0
\def\SRC{\ifSRCOK%
  \ifnum\inputlineno>\LASTLINE%
    \ifnum\LASTLINE<0%
      \global\PAGETOP=\inputlineno%
    \fi%
    \global\LASTLINE=\inputlineno%
    \ifnum\INPSP=0%
      \ifnum\inputlineno>\PAGETOP%
        
      \fi%
    \else%
      
    \fi%
  \fi%
\fi}
\def\PUSH#1{%
\SRC%
\ifnum\INPSP=0 \global\let\INPSTACKA=\CurrentInput \else%
\ifnum\INPSP=1 \global\let\INPSTACKB=\CurrentInput \else%
\ifnum\INPSP=2 \global\let\INPSTACKC=\CurrentInput \else%
\ifnum\INPSP=3 \global\let\INPSTACKD=\CurrentInput \else%
\ifnum\INPSP=4 \global\let\INPSTACKE=\CurrentInput \else%
\ifnum\INPSP=5 \global\let\INPSTACKF=\CurrentInput \else%
               \global\let\INPSTACKX=\CurrentInput \fi\fi\fi\fi\fi\fi%
\gdef\CurrentInput{#1}%
\WinEdt{<+ \CurrentInput}%
\global\LASTLINE=0%
\ifSRCOK\fi%
\global\advance\INPSP by 1}
\def\POP{%
\ifnum\INPSP>0 \global\advance\INPSP by -1  \fi%
\ifnum\INPSP=0 \global\let\CurrentInput=\INPSTACKA \else%
\ifnum\INPSP=1 \global\let\CurrentInput=\INPSTACKB \else%
\ifnum\INPSP=2 \global\let\CurrentInput=\INPSTACKC \else%
\ifnum\INPSP=3 \global\let\CurrentInput=\INPSTACKD \else%
\ifnum\INPSP=4 \global\let\CurrentInput=\INPSTACKE \else%
\ifnum\INPSP=5 \global\let\CurrentInput=\INPSTACKF \else%
               \global\let\CurrentInput=\INPSTACKX \fi\fi\fi\fi\fi\fi%
\WinEdt{<-}%
\global\LASTLINE=\inputlineno%
\global\advance\LASTLINE by -1%
\SRC}
\def\INPUT#1{\relax}

\let\originalxxxeverypar\everypar
\newtoks\everypar
\originalxxxeverypar{\the\everypar\expandafter\SRC}
\everymath\expandafter{\the\everymath\expandafter\SRC}
\output\expandafter{\expandafter\SRCOKfalse\the\output}

\newif\ifSRCOK \SRCOKtrue
\newcount\PAGETOP
\newcount\LASTLINE
\global\PAGETOP=1
\global\LASTLINE=-1
\gdef\MainFile{\jobname.tex}
\gdef\CurrentInput{\MainFile}
\newcount\INPSP
\global\INPSP=0
\def\EJECT{\SRC\eject}
\def\WinEdt#1{\typeout{:#1}}
\def\SRC{\ifSRCOK%
  \ifnum\inputlineno>\LASTLINE%
    \ifnum\LASTLINE<0%
      \global\PAGETOP=\inputlineno%
    \fi%
    \global\LASTLINE=\inputlineno%
    \ifnum\INPSP=0%
      \ifnum\inputlineno>\PAGETOP%
      \fi%
    \else%
    \fi%
  \fi%
\fi}
\def\PUSH#1{%
\SRC%
\ifnum\INPSP=0 \global\let\INPSTACKA=\CurrentInput \else%
\ifnum\INPSP=1 \global\let\INPSTACKB=\CurrentInput \else%
\ifnum\INPSP=2 \global\let\INPSTACKC=\CurrentInput \else%
\ifnum\INPSP=3 \global\let\INPSTACKD=\CurrentInput \else%
\ifnum\INPSP=4 \global\let\INPSTACKE=\CurrentInput \else%
\ifnum\INPSP=5 \global\let\INPSTACKF=\CurrentInput \else%
               \global\let\INPSTACKX=\CurrentInput \fi\fi\fi\fi\fi\fi%
\gdef\CurrentInput{#1}%
\WinEdt{<+ \CurrentInput}%
\global\LASTLINE=0%
\ifSRCOK\fi%
\global\advance\INPSP by 1}
\def\POP{%
\ifnum\INPSP>0 \global\advance\INPSP by -1  \fi%
\ifnum\INPSP=0 \global\let\CurrentInput=\INPSTACKA \else%
\ifnum\INPSP=1 \global\let\CurrentInput=\INPSTACKB \else%
\ifnum\INPSP=2 \global\let\CurrentInput=\INPSTACKC \else%
\ifnum\INPSP=3 \global\let\CurrentInput=\INPSTACKD \else%
\ifnum\INPSP=4 \global\let\CurrentInput=\INPSTACKE \else%
\ifnum\INPSP=5 \global\let\CurrentInput=\INPSTACKF \else%
               \global\let\CurrentInput=\INPSTACKX \fi\fi\fi\fi\fi\fi%
\WinEdt{<-}%
\global\LASTLINE=\inputlineno%
\global\advance\LASTLINE by -1%
\SRC}
\def\INPUT#1{\relax}
\let\OldINCLUDE=\include
\def\include#1{
\EJECT%
\PUSH{#1.tex}%
\OldINCLUDE{#1}%
\POP}

\let\originalxxxeverypar\everypar
\newtoks\everypar
\originalxxxeverypar{\the\everypar\expandafter\SRC}
\everymath\expandafter{\the\everymath\expandafter\SRC}
\let\zzzxxxbibliography=\bibliography
\def\bibliography#1{\PUSH{\jobname.bbl}\zzzxxxbibliography{#1}\POP}
\output\expandafter{\expandafter\SRCOKfalse\the\output}

\begin{document}

\author{Martin Meyries}
\address{Department of Mathematics,
Karlsruhe Institute of Technology, 76128 Karlsruhe, Germany.}
\email{martin.meyries@kit.edu}

\author{Mark Veraar}
\address{Delft Institute of Applied Mathematics\\
Delft University of Technology \\ P.O. Box 5031\\ 2600 GA Delft\\The
Netherlands} \email{M.C.Veraar@tudelft.nl}

\title[Embedding results for weighted function spaces]{Sharp embedding results for spaces of smooth functions with power weights}


\thanks{The first author was supported by the project ME 3848/1-1 of the Deutsche Forschungsgemeinschaft (DFG). The second author was supported by a VENI subsidy 639.031.930 of the Netherlands Organisation for Scientific Research (NWO)}

\maketitle
\begin{abstract}
We consider function spaces of Besov, Triebel-Lizorkin, Bessel-potential and Sobolev type on $\R^d$, equipped with power weights $w(x) = |x|^\gamma$, $\gamma>-d$. We prove two-weight Sobolev embeddings for these spaces. Moreover, we precisely characterize for which parameters the embeddings hold.  The proofs are presented in such a way that they also hold for vector-valued functions.
\end{abstract}

\section{Introduction}

Weighted spaces of smooth functions play an important role in the context of partial differential equations (PDEs).  They are widely used, for instance, to treat PDEs with degenerate coefficients or domains with a nonsmooth geometry (see e.g. \cite{Amann11, Kufner, Maz11, Tr1}). For evolution equations, power weights in time play an important role in order to obtain results for rough initial data (see \cite{CS01, KPW, Lun, PS04}). In addition, here one is naturally confronted with vector-valued spaces. Our work is motivated by this and will be applied in a forthcoming paper in order to study weighted spaces with boundary values.

For general literature on weighted function spaces we refer to  \cite{Bui82, Gri63, Kufner,  Maz11, OpicKufner, Rab, Tri83, Tr1} and references therein. Also vector-valued function spaces are intensively studied (see \cite{Am97, Ama09, Schm,SchmSiunpublished, SchmSi05, Tr97, Zim89} and references therein). Less is known on vector-valued function spaces with weights (see \cite{Amann11, MS11} and references therein). Some difficulties come from the fact that in the vector-valued case the identities $W^{1,p} = H^{1,p}$ and $L^p = F_{p,2}^0$ hold only under further geometric assumptions on the underlying Banach space (see below).

In this paper we characterize continuous embeddings of Sobolev type for vector-valued function spaces with weights of the form $w(x) = |x|^\gamma$ with $\gamma > -d$, where $d$ is the dimension of the underlying Euclidian space. We consider several classes of spaces: Besov spaces, Triebel-Lizorkin spaces, Bessel-potential spaces and Sobolev spaces. In the embeddings which we study we put (possibly different) weights $w_0(x) = |x|^{\gamma_0}$ and  $w_1(x) = |x|^{\gamma_1}$ on each of the function space.

These embeddings and their optimality are well-known in the unweighted case (see e.g. \cite{SchmSiunpublished, SiTr, Tri83}). For scalar-valued Besov spaces with general weights from Muckenhoupt's $A_\infty$-class (see Section \ref{prelim}) the embeddings were characterized in \cite{HaSkr}. In the latter work also the compact embeddings for scalar Triebel-Lizorkin spaces are characterized. Sufficient conditions for scalar-valued Triebel-Lizorkin spaces in the case of one fixed weight $w$ for both spaces are considered in \cite{Bui82}. Results for Sobolev spaces are obtained e.g. in \cite{Kufner, Maz11}. A different setting is studied  \cite{Rab}, which we discuss in Remark \ref{rem:Rab} below.

The approach of \cite{HaSkr} to the scalar Besov space case is based on discretization in terms of wavelet bases and on weighted embeddings of \cite{KuLeSiSk}. In the special case of power weights we can give elementary Fourier analytic proofs for the necessary and sufficient conditions. These apply also in the general vector-valued case, and so we do not have to impose any restriction on the underlying Banach space throughout.

\medskip

For a further discussion, let us describe the results in detail. Throughout, let $X$ be a Banach space. For $p\in (1, \infty]$ and $q\in [1, \infty]$, let $B^{s}_{p,q}(\R^d,w;X)$ denote the Besov space with weight $w(x) = |x|^{\gamma}$, where $\gamma>-d$ (see Section \ref{subsec:defBF}).
The following two-weight characterization of Sobolev type embeddings for these spaces is the first main result of our paper.
\begin{theorem}\label{thm:main1}
Let $X$ be a Banach space, $1<p_0,p_1\leq \infty$, $q_0,q_1\in [1,\infty]$, $s_0,s_1\in \R$, and $w_0(x) = |x|^{\gamma_0}$, $w_1(x) = |x|^{\gamma_1}$ with  $\gamma_0,\gamma_1>-d$.
The following assertions are equivalent:
\begin{enumerate}[(1)]
\item One has the continuous embedding
\begin{equation}\label{eq:introBesov}
B^{s_0}_{p_0, q_0}(\R^d,w_0; X)\hookrightarrow B^{s_1}_{p_1, q_1}(\R^d,w_1; X).
\end{equation}
\item The parameters satisfy one of the following conditions:
\begin{equation}\label{cond:trivial}
 \gamma_0 = \gamma_1, \qquad p_0 = p_1 \quad \text{and either}\quad s_0 > s_1 \quad \text{or} \quad s_0 = s_1\quad \text{and} \quad q_0\leq q_1;
\end{equation}
 \begin{equation}\label{cond:exp}
\frac{\gamma_1}{p_1}\leq \frac{\gamma_0}{p_0}, \qquad  \frac{d+\gamma_1}{p_1}< \frac{d+\gamma_0}{p_0} \qquad \text{and} \qquad s_0 -\frac{d+\gamma_0}{p_0} > s_1 - \frac{d+\gamma_1}{p_1};
\end{equation}
\begin{equation}\label{cond:expnew}
\frac{\gamma_1}{p_1}\leq \frac{\gamma_0}{p_0}, \qquad  \frac{d+\gamma_1}{p_1}< \frac{d+\gamma_0}{p_0}, \qquad q_0 \leq q_1\qquad \text{and} \qquad s_0 -\frac{d+\gamma_0}{p_0} =  s_1 - \frac{d+\gamma_1}{p_1}.
\end{equation}
\end{enumerate}
\end{theorem}

For $p\in (1, \infty)$, $q\in [1, \infty]$ and $w$ as above, let $F^{s}_{p,q}(\R^d,w;X)$  denote the weighted Triebel--Lizorkin space (see Section \ref{subsec:defBF}). The following characterization is our second main result. Unlike in Theorem \ref{thm:main1}, the characterization is given only for $1<p_0\leq p_1<\infty$ (see, however, Proposition \ref{prop:p1kleinerp0} below). Important in this result is that $q_0\leq q_1$ is not required in the sharp case $s_0 -\frac{d+\gamma_0}{p_0} =  s_1 - \frac{d+\gamma_1}{p_1}$ as for Besov spaces.

\begin{theorem}\label{thm:main2}
Let $X$ be a Banach space, $1<p_0\leq p_1<\infty$, $q_0,q_1\in [1,\infty]$, $s_0,s_1\in \R$, and $w_0(x) = |x|^{\gamma_0}$, $w_1(x) = |x|^{\gamma_1}$ with  $\gamma_0,\gamma_1>-d$. The following assertions are equivalent:
\begin{enumerate}[(1)]
\item One has the continuous embedding
\begin{equation}\label{eq:introF}
F^{s_0}_{p_0, q_0}(\R^d,w_0; X)\hookrightarrow F^{s_1}_{p_1, q_1}(\R^d,w_1; X).
\end{equation}
\item The parameters satisfy either \eqref{cond:trivial} or
\begin{equation}\label{cond:expF}
 \frac{\gamma_1}{p_1}\leq \frac{\gamma_0}{p_0}, \qquad  \frac{d+\gamma_1}{p_1}< \frac{d+\gamma_0}{p_0} \qquad \text{and} \qquad s_0 -\frac{d+\gamma_0}{p_0} \geq   s_1 - \frac{d+\gamma_1}{p_1}.
\end{equation}

\end{enumerate}
\end{theorem}

\begin{remark}
\
\begin{enumerate}[(i)]

\item The scalar version $X=\C$ of Theorem \ref{thm:main1} with  general weights $w_0,w_1$ from Muckenhoupt's $A_\infty$-class is proved in \cite[Section 2]{HaSkr}. For $w_0 = w_1\in A_\infty$ satisfying $\inf_{x\in \R^d} w(B(x,t))\geq t^\varepsilon$ with $\varepsilon>0$, the implication (2) $\Rightarrow$ (1) of Theorem \ref{thm:main2} can be found in \cite[Theorem 2.6]{Bui82} in the scalar case. In our setting, this corresponds to the case $\gamma_0 = \gamma_1 \geq 0$.

\item In the unweighted case, i.e., $\gamma_0 = \gamma_1= 0$, results such as Theorems \ref{thm:main1} and \ref{thm:main2} are well-known and go back to works of Jawerth, Nikol'skij, Peetre and Triebel (see \cite[Section 2.7.1]{Tri83} for a historical overview). A detailed account on these embeddings in the vector-valued setting can be found in \cite{SchmSiunpublished}.

\item Theorem \ref{thm:main1} gives embeddings for $p_0 >p_1$, which is only possible in the presence of weights. In  Proposition \ref{prop:p1kleinerp0} we obtain a partial result also for Triebel-Lizorkin spaces in this case.

\item In Theorems \ref{thm:main1} and \ref{thm:main2}, suppose that  $p_0 < p_1$. Then the condition $\frac{d+\gamma_1}{p_1}< \frac{d+\gamma_0}{p_0}$ in \eqref{cond:exp}, \eqref{cond:expnew} and \eqref{cond:expF} is redundant. Similiarly, if $p_0 > p_1$ then $\frac{\gamma_1}{p_1}\leq \frac{\gamma_0}{p_0}$ is redundant.

\item It follows from $p_0 \leq p_1$ and $\frac{\gamma_1}{p_1} \leq \frac{\gamma_0}{p_0}$ that $\gamma_0 = \gamma_1<0$ is excluded. In this case one only has the trivial embeddings (i.e.,\ the embeddings under the assumption  \eqref{cond:trivial}).

\item It is a well-known fact that Sobolev embeddings for Triebel-Lizorkin spaces are independent of the microscopic parameters $q_0,q_1$ (see \cite{Bui82, SchmSiunpublished, Tri83}).

\end{enumerate}
\end{remark}

Our proof of the sufficiency of the stated relations in Theorem \ref{thm:main1} is based on a direct two-weight extension of an inequality of Plancherel-Polya-Nikol'skij type (see Proposition \ref{prop:Nikolskii}, and \cite[Section 1.3]{Tri83} for an overview). For $\gamma_0,\gamma_1 \geq 0$ this inequality is obtained by extending the proof of the one-weight version of \cite{Bui82} to the present situation. For negative weight exponents we use the weighted Young inequalities from \cite{Bui94, Kerman}. The necessity of these conditions follows from suitable scaling arguments, see Propositions \ref{prop:necc} and \ref{prop:neccnonopt2}. Observe that $L^{p}(\R^d, |\cdot|^{\gamma})$ scales to the power $-\frac{d+\gamma}{p}$, which explains the importance of this number. Moreover, the relation $\frac{\gamma_1}{p_1}\leq \frac{\gamma_0}{p_0}$ is in particular sufficient to apply the results of \cite{Bui94, Kerman}.

Theorem \ref{thm:main2} is derived from Theorem \ref{thm:main1} using a weighted version of a Gagliardo-Nirenberg type inequality for $F$-spaces (see Proposition \ref{prop:interpolationineq}). Here we follow the presentation of \cite{SchmSiunpublished}.

\medskip

As a consequence of Theorem \ref{thm:main2} we characterize embeddings for Bessel-potential and Sobolev spaces. For $p\in (1, \infty)$ and $w(x) = |x|^{\gamma}$, where $\gamma\in (-d, d(p-1))$, let $H^{s,p}(\R^d,w;X)$ denote the weighted Bessel-potential space with $s\in \R$, and let $W^{m,p}(\R^d,w;X)$ denote the weighted Sobolev space with $m\in \N_0$ (see Section \ref{sub:Bessel}).
\begin{corollary}\label{cor:introHspaces}
Let $X$ be a Banach space, $1<p_0\leq p_1< \infty$, $s_0,s_1\in \R$, and $w_0(x) = |x|^{\gamma_0}$, $w_1(x) = |x|^{\gamma_1}$ with $\gamma_0\in (-d,d(p_0-1)),\gamma_1\in (-d,d(p_1-1))$. The following assertions are equivalent:
\begin{enumerate}[(1)]
\item One has the continuous embedding
\[
H^{s_0,p_0}(\R^d,w_0;X )\hookrightarrow H^{s_1,p_1}(\R^d,w_1;X).
\]
\item The parameters satisfy
\begin{equation}\label{cond:exp2}
\frac{\gamma_1}{p_1}\leq \frac{\gamma_0}{p_0}\qquad \text{and} \qquad s_0 -\frac{d+\gamma_0}{p_0}\geq s_1 - \frac{d+\gamma_1}{p_1}.
\end{equation}
\end{enumerate}
\end{corollary}

\begin{corollary}\label{cor:introWspaces}
Let $X$ be a Banach space, $1<p_0\leq p_1< \infty$, $s_0,s_1\in \N_0$, and $w_0(x) = |x|^{\gamma_0}$, $w_1(x) = |x|^{\gamma_1}$ with $\gamma_0\in (-d,d(p_0-1)),\gamma_1\in (-d,d(p_1-1))$. The following assertions are equivalent:
\begin{enumerate}[(1)]
\item One has the continuous embedding
\[W^{s_0,p_0}(\R^d,w_0;X )\hookrightarrow W^{s_1,p_1}(\R^d,w_1;X).\]
\item The parameters satisfy  \eqref{cond:exp2}.
\end{enumerate}
\end{corollary}

The necessity of \eqref{cond:exp2} for the embeddings in the Corollaries \ref{cor:introHspaces} and \ref{cor:introWspaces} is actually valid for all $\gamma_0,\gamma_1 > -d$, as a consequence of Proposition \ref{prop:necc}. The restrictions in the sufficiency part mean that $w_0 \in A_{p_0}$ and $w_1\in A_{p_1}$, where $A_p$ denotes Muckenhoupt's class (see Section \ref{prelim}).

In the general vector-valued case, the $H$- and the integer $W$-spaces are not contained in the $B$- and $F$-scale, respectively. It holds that
\begin{equation}\label{H=F}
 H^{s,p}(\R^d;X) = F^{s}_{p,2}(\R^d;X), \qquad \text{for some }s\in \R, \quad p\in (1, \infty),
\end{equation}
if and only if $X$ is isomorphic to a Hilbert space (see \cite{HaMe96} and \cite[Remark 7]{SchmSiunpublished}). Moreover,
\begin{equation}\label{H=W}
 H^{1,p}(\R^d;X) = W^{1,p}(\R^d;X) \qquad \text{for some }p\in (1, \infty)
\end{equation}
characterizes the UMD property of $X$ (see \cite{Ama95,McC84,Zim89} for details).

We obtain the sufficient conditions in the above corollaries from embeddings of type \begin{equation}\label{squezzeintro}
F_{p,1}^s(\R^d,w;X) \hookrightarrow H^{s,p}(\R^d,w;X)\hookrightarrow F_{p,\infty}^s(\R^d,w;X),
\end{equation}
which are valid for all Banach spaces $X$ and $A_p$-weights $w$ (see Proposition \ref{prop:squeezeHF}), and the independence of the embeddings for the $F$-spaces. For the latter embedding in \eqref{squezzeintro} we show in Remark \ref{remarkProp3.12} that it is necessary that $w$ satisfies a local $A_p$-condition, which results in  $\gamma_0<d(p_0-1)$. Then $\gamma_1<d(p_1-1)$ is a consequence of $\gamma_1/p_1\leq \gamma_0/ p_0$ and $p_0\leq p_1$. This explains our restrictions on the weight exponents to the $A_p$-case. In view of the Theorems \ref{thm:main1} and \ref{thm:main2}, and the results in \cite{Rab} (see Remark \ref{rem:Rab}), we do not expect these restrictions to be necessary.

\medskip

The result of Theorem \ref{thm:main1} for Besov spaces and $1<p_1<p_0<\infty$ cannot be extended to $F$-, $H$- and $W$-spaces in the usual way. We prove the following  in this case.
\begin{proposition}\label{prop:p1kleinerp0}
Let $X$ be a Banach space, $1<p_1<p_0 < \infty$, $s_0,s_1\in \R$ and $\gamma_0\in (-d,d(p_0-1)),\gamma_1\in (-d,d(p_1-1))$, and let $w_0(x) = |x|^{\gamma_0}$ and $w_1(x) = |x|^{\gamma_1}$. The following assertions are equivalent:
\begin{enumerate}[(1)]
\item One has the continuous embedding
\[
H^{s_0,p_0}(\R^d,w_0;X )\hookrightarrow H^{s_1,p_1}(\R^d,w_1;X).
\]

\item One has the continuous embedding
\[
F^{s_0}_{p_0,2}(\R^d,w_0;X )\hookrightarrow F^{s_1}_{p_1,2}(\R^d,w_1;X).
\]

\item The parameters satisfy
\begin{equation}\label{eq:cond3}
\frac{d+\gamma_1}{p_1}< \frac{d+\gamma_0}{p_0} \qquad \text{and} \qquad s_0 -\frac{d+\gamma_0}{p_0}>s_1 -
\frac{d+\gamma_1}{p_1}.
\end{equation}
\end{enumerate}
\end{proposition}

Note that the above result shows that there is no embedding in the important sharp case $s_0 -\frac{d+\gamma_0}{p_0}=s_1 - \frac{d+\gamma_1}{p_1}$. Surprisingly, this is different for Besov spaces (see Theorem \ref{thm:main1}). The same holds if the $H$-spaces are replaced by $W$-spaces since in the scalar case $W^{m,p}(\R^d,w) = H^{m,p}(\R^d,w)$ whenever $m\in \N$ and $w\in A_p$ (see the proof of Corollary \ref{cor:Wcasecounter}).

\begin{remark}
Let $p_0>p_1$, $w_0(x) = |x|^{\gamma_0}$ and $w_1(x) =  |x|^{\gamma_1}$ with $\gamma_0, \gamma_1>-d$.
It would be interesting to characterize for which parameters one has $F^{s_0}_{p_0,q_0}(\R^d, w_0) \hookrightarrow F^{s_1}_{p_1,q_1}(\R^d,w_1)$.
Proposition \ref{prop:p1kleinerp0} only contains a partial answer to this, because of the restrictions on $\gamma_0,\gamma_1, q_0$ and $q_1$.
\end{remark}

\begin{remark}Let $J_{s} = (1-\Delta)^{s/2}$ be the Bessel potential. Corollary \ref{cor:introHspaces} gives a characterization of the boundedness of $J_{-s_0}:L^{p_0}(\R^d,w_0)\to L^{p_1}(\R^d,w_1)$ for certain power weights $w_0$ and $w_1$ for $p_0\leq p_1$ and all $s_0\geq 0$. The equivalent condition of \cite{Rak97} can be difficult to check. Moreover, in Proposition \ref{prop:p1kleinerp0} we provide a characterization of the boundedness of $J_{-s_0}$ for power weights $w_0$ and $w_1$ for $p_1< p_0$ and all $s_0\geq 0$. It seems that the problem for $J_{-s}$ with $p_1<p_0$ has not been considered before.

In \cite{SawWhe92} necessary and sufficient conditions on $w_0, w_1, p_0, p_1, s_0$ can be found for the boundedness of the Riesz potential $I_{-s_0} = (-\Delta)^{-s_0/2}$ for $p_0\leq p_1$ and $0\leq s_0<d$. The case $p_1<p_0$ has been considered in \cite{Verb}.
\end{remark}

Together with elementary embeddings, the Theorems \ref{thm:main1} and \ref{thm:main2} provide certain embeddings between $B$- and $F$-spaces. These can be strengthened to so-called Jawerth-Franke embeddings, see Theorem \ref{thm:sobolevJawerth}. Here we restrict to the case of $A_p$-weights (see Remark \ref{discJF} for a discussion). We also use the above results to show embeddings of weighted $B$-, $F$-, $H$- and $W$-spaces into (unweighted) spaces of continuous functions, see Proposition \ref{embedding-C}. Again we can follow the presentation of \cite{SchmSiunpublished}. For the case of a single weight $w\in A_\infty$, in the scalar-valued case the Jawerth-Franke embeddings are shown in \cite{Bui82, HaSkr}.

\begin{remark}\label{rem:Rab}
Recently, in \cite{Rab} the classical Caffarelli--Kohn-Nirenberg inequalities (see \cite{CKN84}) are extended to general power weights. It is  characterized when
\begin{equation}\label{eq:Rabieremb}
W^{1,p_0,q}(\R^d\setminus\{0\}, w_0,w)\hookrightarrow L^{p_1}(\R^d, w_1),
\end{equation}
   where $W^{1,p_0,q}(\R^d\setminus\{0\}, w_0,w)$ is the space of functions $f:\R^d\setminus \{0\}\to \R$ which are locally integrable and satisfy
\[\|f\|_{L^{p_0}(w_0)}  + \|\nabla f\|_{L^q(w)} <\infty.\]
Here it is assumed that $w_0(x) = |x|^{\gamma_0}$,  $w(x) = |x|^{\gamma}$ and $w_1(x) = |x|^{\gamma_1}$ with $\gamma_0, \gamma,\gamma_1\in \R$ and $p_0,q,p_1\in [1, \infty)$. Our results can be compared to the ones in \cite{Rab} only in a very special case of both our and the setting in \cite{Rab}, namely for $\gamma_0 = \gamma\in (-d,d(p_0-1))$, $\gamma_1\in (-d,d(p_1-1))$, $q=p_0, p_1<\infty$, $s_0=1$ and $s_1=0$. In this special case the equivalence of \eqref{eq:Rabieremb} with the condition on the parameters $\gamma_0, \gamma_1, p_0, p_1$ in \cite{Rab} coincides with ours (see Corollary \ref{cor:introWspaces} if $p_0\leq p_1$ and Proposition \ref{prop:p1kleinerp0}  if $p_0>p_1$).

On the other hand, one of the main and novel points in the results in \cite{Rab} is that the powers weights are not necessarily of $A_{\infty}$-type and two different powers weights $w_0,w$ and exponents $p_0, q$ in the Sobolev space are considered.
\end{remark}

\begin{remark}
It would be interesting to extend the above results (e.g.\ Theorems \ref{thm:main1}, \ref{thm:main2}, etc.)\ to homogeneous Besov, Triebel-Lizorkin, Bessel-potential spaces and Sobolev spaces.
Some parts are easy to extend. In particular, Proposition \ref{prop:Nikolskii} below can be applied in the homogeneous setting as well. The identity $s_0 -\frac{d+\gamma_0}{p_0} = s_1 -\frac{d+\gamma_1}{p_1}$ plays a crucial role.
\end{remark}

\textbf{Outline.} In Section 2 we consider preliminaries for the treatment of weighted function spaces, such as maximal inequalities and a multiplier theorems. These are used in Section 3 to derive basic properties of the vector-valued spaces. In Section 4 we prove the characterization for Besov spaces, and in Section 5 the case of $F$-, $H$- and $W$ spaces is treated. In Section 6 we show embeddings of Jawerth-Franke type and in Section 7 embeddings  into unweighted function spaces. \medskip

\textbf{Notations.} Positive constants are denoted by $C$, and may vary from
line to line.  If $X,Y$ are Banach spaces, we write $X=Y$ if they coincide as
sets and have equivalent norms. For $p\in [1,\infty]$, the standard sequence
spaces are denoted by $\ell^p$.
\medskip

\textbf{Acknowledgement.} The authors thank Dorothee D. Haroske and Winfried Sickel for bringing the paper \cite{HaSkr} to their attention.

\section{Preliminaries}\label{prelim}

Here we collect the tools from harmonic analysis that are needed in the treatment of the weighted function spaces in the next section.

\subsection{Weights} A function $w:\R^d\to [0,\infty)$ is called a {\em weight} if $w$ is locally integrable
and if $\{w=0\}$ has Lebesgue measure zero. For $p\in[1,\infty)$ we denote by
$A_p$ the \emph{Muckenhoupt class} of weights, and $A_\infty =
\bigcup_{1\leq p< \infty} A_{p}$. In case $p\in (1,\infty)$ we have $w\in A_p$ if
$$\sup_{Q \text{ cubes in }\R^d} \left ( \frac{1}{|Q|} \int_Q w(x)\, dx\right) \left ( \frac{1}{|Q|} \int_Q w(x)^{-\frac{1}{p-1}} \, dx\right)^{p-1} <\infty.$$
We refer to \cite{GCRdF}, \cite[Chapter 9]{GraModern} and \cite[Chapter V]{Stein93} for
the general properties of these classes.

\begin{example}
Let $w$ be a power weight, i.e., $w(x) = |x|^{\gamma}$ with $\gamma\in \R$.
Then for $p\in (1,\infty)$ we have (see \cite[Example 9.1.7]{GraModern})
$$w\in A_p \quad \text{if and only if}\quad \gamma\in (-d, d(p-1)).$$
\end{example}

Let $(X,\|\cdot\|)$ be a Banach space. For a strongly measurable
function $f:\R^d\to X$ and $p\in [1, \infty)$ let
\[\|f\|_{L^p(\R^d,w;X)} := \Big(\int_{\R^d} \|f(x)\|^p w(x)\, dx\Big)^{1/p}.\]
For $p\in [1, \infty)$ we consider the Banach space
\[L^p(\R^d,w;X) := \{f\text{ strongly
measurable}\,:\,  \|f\|_{L^p(\R^d,w;X)}<\infty\},\] and set further
$L^\infty(\R^d,w;X) := L^\infty(\R^d;X)$.

\subsection{Maximal functions}
For $f\in L^1_{\text{loc}}(\R^d;X)$, let the Hardy-Littlewood maximal function $Mf$ be defined by
\[ (M f) (x) = \sup_{r>0} \frac{1}{|B(x,r)|} \int_{B(x,r)} \|f(y)\| \, dy, \qquad x\in \R^d,\]
where $B(x,r) = \{z\in \R^d: |z-x|<r\}$. For any weight $w:\R^d\to \R_+$ and $p\in (1, \infty)$, the maximal function $M$ is bounded on $L^p(\R^d,w)$ if and only if $w\in A_p$ (see \cite[Theorem V.3.1]{Stein93} and \cite[Theorem 9.1.9]{GraModern}). The following weighted vector-valued version of the Fefferman-Stein maximal inequality holds as well.
\begin{proposition}\label{prop:FSvectorvalued}
Let $X$ be a Banach space, $p\in (1, \infty)$, $q\in (1, \infty]$ and $w\in A_p$. Then there exists a constant $C_{p,q,w}$ such that for all $(f_k)_{k\geq 0} \subset L^p(\R^d,w;X)$ one has
\[\|(M f_k)_{k\geq 0}\|_{L^p(\R^d,w;\ell^q)}\leq C_{p,q,w} \|(f_k)_{k\geq 0}\|_{L^p(\R^d,w;\ell^q(X))}.\]
\end{proposition}
\begin{proof}
For $q=\infty$ one uses that
\begin{equation}\label{2}
 \|M f_k(x)\|_{\ell^\infty(X)}\leq M \|f_k(x)\|_{\ell^\infty(X)}\qquad x\in \R^d, \qquad k\geq 0,
\end{equation}
and applies the boundedness of $M$ on $L^p(\R^d,w)$ to the function $f(x) = \|f_k(x)\|_{\ell^\infty(X)}$. If $1<q<\infty$ and $X= \R$, the result can be found in \cite[Theorem 3.1]{AJ80} and \cite{Koki}. The vector-valued case can be obtained from the scalar case by applying it to  $(\|f_k\|)_{k\geq 0}   \subset L^p(\R^d,w)$.
\end{proof}

For a given function $\varphi:\R^d\to \C$ and $t>0$, we define the function $\varphi_t:\R^d\to \C$ by
\[\varphi_{t}(x) := t^{n} \varphi(t x)\]
The following lemma is well-known to experts.
\begin{lemma}\label{lem:Stein}
Let $X$ be a Banach space, $p\in (1, \infty)$ and $w\in A_p$. For $\varphi\in L^1(\R^d)$, define $\psi:\R^d\to \R$ by
\[\psi(x) := \sup\{|\varphi(y)|: |y|\geq |x|\}.\]
If $\psi\in L^1(\R^d)$, then there is a constant $C_{p,w}$ such that for all $f\in L^p(\R^d,w;X)$ one has
\[\Big\|\sup_{t>0}\|\varphi_{t}* f\| \Big\|_{L^p(\R^d,w)}\leq C_{p,w}\|\psi\|_{L^1(\R^d)} \| f\|_{L^p(\R^d,w;X)}.\]
Moreover, for $q\in (1,\infty]$ there is a constant $C_{p,q,w}$ such that for all $(t_k)_{k\geq 0} \subset \R_+$ and $(f_k)_{k \geq 0} \subset L^p(\R^d,w;\ell^q(X))$ it holds
$$\|(\varphi_{t_k}* f_k)_{k\geq 0}\|_{L^p(\R^d,w; \ell^q(X))}\leq C_{p,q,w}\|\psi\|_{L^1(\R^d)} \|(f_k)_{k \geq 0}\|_{L^p(\R^d,w; \ell^q(X))}.$$
\end{lemma}
A partial converse of the lemma holds as well: if for a positive radial decreasing function $\varphi\neq 0$, one has $\|\varphi_{t}* f\|_{L^p(\R^d,w)}\leq C \|f\|_{L^p(\R^d,w)}$ with a constant $C$ independent of $f$ and $t>0$, then $w\in A_p$
(see \cite[p. 198]{Stein93}).
\begin{proof}[Proof of Lemma \ref{lem:Stein}]
Let $g(x) = \|f(x)\|$. By definition of $\psi$ and \cite[Theorem 2.1.10]{GraClass}, for  $x\in \R^d$ and $t>0$ one has
\begin{align}
\label{eq:Stein}\|\varphi_{t}* f(x)\|& \leq \int_{\R^d} |\varphi_t(y)| \, \|f(x-y)\| \, dy \leq \int_{\R^d} \psi_t(y) \, g(x-y) \, dy\\
& \leq \|\psi_t\|_{L^1(\R^d)} M g(x) =\|\psi\|_{L^1(\R^d)} M g(x).\nonumber
\end{align}
Now the first asserted inequality follows from the boundedness of $M$ on $L^p(\R^d,w)$. Further, for all $q\in (1,\infty]$ inequality \eqref{eq:Stein} implies that
$$\|(\varphi_{t_k}* f_k(x))_{k\geq 0}\|_{\ell^q(X)}\leq \|\psi\|_{L^1(\R^d)} \|(M f_k(x))_{k\geq 0}\|_{\ell^q(X)}, \qquad x\in \R^d.$$
Thus the second assertion for $q\in (1,\infty)$ follows from Proposition \ref{prop:FSvectorvalued}. The case $q=\infty$ is a consequence of \eqref{2} and the boundedness of $M$ on $L^p(\R^d,w)$.
\end{proof}

\subsection{A multiplier theorem} Let $X$ be a Banach space, and let $\Schw(\R^d;X)$ be the space of $X$-valued Schwartz functions. We write $\Schw(\R^d)$ in the scalar case $X=\C$. The Fourier transform of a function $f\in \Schw(\R^d;X)$ is given by
\[ \F f(\xi) =  \wh{f}(\xi) =    \frac1{(2\pi)^{d/2}}\int_{\R^d} f(x)e^{-ix\cdot\xi}\,dx, \qquad \xi\in\R^d.\]
Recall that $\F$ is a continuous isomorphism on  $\Schw(\R^d;X)$. A linear map $f:\Schw(\R^d)\to X$ is called an \emph{$X$-valued tempered distribution}, if for all $\phi\in \Schw(\R^d)$, there are a constant $C$ and $k, N\in \N$  such that
\[\|f(\phi)\|\leq C \sup_{x\in \R^d}\sup_{|\alpha|\leq N} (1+|x|)^k |D^{\alpha}\phi(x)|.\]
The space of all $X$-valued tempered distributions is denoted  by $\TD(\R^d;X)$. Standard operators (Fourier transform, convolution, etc.) on $\TD(\R^d;X)$ can be defined as in the scalar case, cf. \cite[Section III.4]{Ama95}.

The following result is an extension of \cite[Theorem 1.6.3]{Tri83} and \cite[Formula 15.3(iv)]{Tr97} to the weighted vector-valued setting. For a compact set $K\subseteq \R^d$, let $L^p_K(\R^d,w)$ be the space of functions $f\in L^p(\R^d,w)$ with $\supp(\wh{f})\subseteq K$ in the sense of distributions.

\begin{proposition}\label{prop:weightedmultiplier}
Let $X$ be a Banach space, $p\in [1, \infty)$, $q\in [1, \infty]$, $r\in (0,\min\{p,q\})$ and $w\in A_{p/r}$.
Let $K_0, K_1, \ldots \subset \R^d$ be compact sets with $\theta_n = \text{diam}(K_n)>0$ for all $n$. Then there is a constant $C_{p,q,r,w}$ such that for all $(m_n)_{n\geq 0} \subset L^\infty(\R^d)$ and $(f_n)_{n\geq 0} \subset L^p_{K_n}(\R^d,w;X)$  one has
\begin{align*}
\|(\F^{-1} & [m_n \wh{f}_n])_{n\geq 0}\|_{L^p(\R^d,w;\ell^q(X))} \\ & \leq C_{p,q,d,r,w}\, \sup_{k\geq 0} \|(1+|\cdot|^{d/r}) \F^{-1}[m_k(\theta_k \cdot)]\|_{L^1(\R^d)} \|(f_n)_{n\geq 0}\|_{L^p(\R^d,w;\ell^q(X))}
\\ & \leq C_{p,q,d,r,w,\lambda} \,\sup_{k\geq 0} \|m_k(\theta_k \cdot)\|_{H^{\lambda,2}(\R^d)} \|(f_n)_{n\geq 0}\|_{L^p(\R^d,w;\ell^q(X))},
\end{align*}
where $\lambda> \frac{d}{2} + \frac{d}{r}$.
\end{proposition}
By considering only $m_1 \neq 0$ and $f_1\neq 0$, one obtains a multiplier theorem on $L^p_{K_1}(\R^d)$.

\begin{proof}
This can be proved as in \cite[Theorem 1.6.3]{Tri83} (see also \cite[Section 15.3]{Tr97}). Indeed, using Proposition \ref{prop:FSvectorvalued} one can extend \cite[Theorem 1.6.2]{Tri83} to the vector-valued setting with weights in $A_{p/r}$. Now the argument in \cite[Theorem 1.6.3]{Tri83} can be repeated to obtain the result.
\end{proof}

\section{Definitions and properties of weighted function spaces}\label{defpropspaces}

In this section we define and investigate the basic properties of function spaces with general $A_\infty$-weights. The definitions extend those of \cite{Bui82} to the vector-valued setting. In \cite{Ry01} it is shown that one actually only needs a local version of this condition, called $A_\infty^{\rm loc}$, to obtain reasonable spaces. Since we are mainly interested in power weights $w(x) = |x|^{\gamma}$, for which one easily sees that $w\in A_\infty$ if and only if $w\in A_{\infty}^{{\rm loc}}$, we restrict ourselves to the $A_\infty$-case.

The arguments employed for the basic properties of the spaces are well-known and completely analogous to the unweighted, scalar-valued case (see e.g. \cite{SchmSiunpublished, Tri83}). We thus refrain from giving too many details and rather refer to the literature at most of the points.

\subsection{Besov and Triebel-Lizorkin spaces\label{subsec:defBF}}

\begin{definition}\label{def:Phi}
Let $\varphi \in \Schw(\R^d)$ be such that
\begin{equation}\label{eq:propvarphi}
0\leq \wh{\varphi}(\xi)\leq 1, \quad  \xi\in \R^d, \qquad  \wh{\varphi}(\xi) = 1 \ \text{ if } \ |\xi|\leq 1, \qquad  \wh{\varphi}(\xi)=0 \ \text{ if } \ |\xi|\geq \frac32.
\end{equation}
Let $\wh{\varphi}_0 = \wh{\varphi}$, $\wh{\varphi}_1(\xi) = \wh{\varphi}(\xi/2) - \wh{\varphi}(\xi)$ and
\[\wh{\varphi}_k(\xi) = \wh{\varphi}_1(2^{-k+1} \xi) = \wh{\varphi}(2^{-k}\xi) - \wh{\varphi}(2^{-k+1}\xi),  \qquad \xi\in \R^d, \qquad  k\geq 1.\]
Let $\Phi$ be the set of all sequences $(\varphi_n)_{n\geq 0}$ constructed in the above way from a function $\varphi$ that satisfies \eqref{eq:propvarphi}.
\end{definition}

For  $\varphi$ as in the definition and $f\in \TD(\R^d;X)$ one sets
\[S_k f := \varphi_k * f = \F^{-1} [\wh{\varphi}_k \wh{f}],\]
which belongs to $C^\infty(\R^d;X)\cap \TD(\R^d;X)$ (see \cite[Remark 4.3.3]{Schm}). Since $\sum_{k\geq 0} \wh{\varphi}_k(\xi) = 1$ for all $\xi\in \R^d$, we have $\sum_{k\geq 0} S_kf =  f$ in the sense of distributions.

\begin{definition}\label{def:Besov}
Let $X$ be a Banach space, $p,q\in [1,\infty]$, $s\in\R$ and $w\in A_{\infty}$. The {\em Besov space} $B_{p,q}^s(\R^d,w;X)$ is defined as the space of all
$f\in {\mathscr S}'(\R^d;X)$ for which
\[ \|f\|_{B_{p,q}^s (\R^d,w;X)} := \Big\| \big( 2^{ks}S_k f\big)_{k\geq 0} \Big\|_{\ell^q(L^p(\R^d,w;X))} < \infty.\]
Moreover, if $s\in \R_+ \backslash \N$, then we set
$$W^{s,p}(\R^d,w;X) := B_{p,p}^s (\R^d,w;X).$$
\end{definition}

\begin{definition}\label{def:Triebel-Lizorkin}
Let $X$ be a Banach space, $p\in [1,\infty)$, $q\in [1,\infty]$, $s\in\R$ and $w\in A_{\infty}$. The {\em Triebel-Lizorkin space} $F_{p,q}^s(\R^d,w;X)$ is defined as the space of all $f\in {\mathscr S}'(\R^d;X)$ for which
\[ \|f\|_{F_{p,q}^s (\R^d,w;X)} := \Big\| \big( 2^{ks}S_k f\big)_{k\ge 0} \Big\|_{L^p(\R^d,w;\ell^q(X))}<\infty.\]
\end{definition}

If $w\equiv 1$, we write $B_{p,q}^s(\R^d;X)$ for $B_{p,q}^s(\R^d,w;X)$ and $F_{p,q}^s(\R^d;X)$ for $F_{p,q}^s(\R^d,w;X)$. As in the scalar case, one can show that these are Banach spaces.
Note that
\begin{align*}
\|f\|_{B_{p,q}^s (\R^d,w;X)} &= \Big(\sum_{k\geq 0} 2^{ks}\|S_kf\|_{L^p(\R^d,w;X)}^q\Big)^{1/q},\\
\|f\|_{F_{p,q}^s (\R^d,w;X)}  &  = \Big\|\Big(\sum_{k\geq 0} 2^{ks}\|S_kf\|^q\Big)^{1/q}\Big\|_{L^p(\R^d,w)},
\end{align*}
with the usual modifications for $q = \infty$.

The following is a consequence of Proposition \ref{prop:weightedmultiplier}.
\begin{proposition}[Independence of $\varphi$]\label{prop:independencevarphi}
Let $X$ be a Banach space and $w\in A_\infty$.
\begin{enumerate}[(1)]
\item For all $s\in \R$ and $p,q\in [1, \infty]$, the space $B^{s}_{p,q}(\R^d,w;X)$ is independent of the choice $(\varphi_n)_{n\geq 0}\in \Phi$. Any  $(\psi_n)_{n\geq 0}\in \Phi$ leads to an equivalent norm in $B^{s}_{p,q}(\R^d,w;X)$.
\item For all $s\in \R$, $p\in [1, \infty)$ and $q\in [1, \infty]$, the space $F^{s}_{p,q}(\R^d,w;X)$ is independent of the choice $(\varphi_n)_{n\geq 0}\in \Phi$. Any $(\psi_n)_{n\geq 0}\in \Phi$ leads to an equivalent norm in $F^{s}_{p,q}(\R^d,w;X)$.
\end{enumerate}
\end{proposition}
\begin{proof} Choose $r\in (0,\min\{p,q\})$ such that $w\in A_{p/r}$. Then one can apply Proposition \ref{prop:weightedmultiplier} in the same way as in the unweighted case in \cite[Section 2.3.2]{Tri83} and \cite[Section 15.5]{Tr97}.
\end{proof}

\begin{remark}\
\begin{enumerate}[(i)]
\item At a technical point in the proofs, $w\in A_{\infty}$ is required to have the boundedness of the Hardy-Littlewood maximal function and of the Fefferman-Stein maximal function in some $L^r$-space and some $L^r(\ell^q)$-space with $r\in (1,\infty)$, respectively.

\item Definition \ref{def:Triebel-Lizorkin} can be extended to $p=\infty$ and $q\in [1, \infty]$, but Proposition \ref{prop:independencevarphi} is not true
in this setting (see \cite[Remark 2.3.1/4]{Tri83}).

\item Rychkov \cite{Ry01} considers scalar $B$- and $F$-spaces with more general weights of class $A_\infty^\text{loc}$, i.e., satisfying only a local $A_\infty$ condition.
\end{enumerate}
\end{remark}

\subsection{Sobolev and Bessel-potential spaces\label{sub:Bessel}}
\begin{definition}
Let $X$ be a Banach space, $m\in \N$, $p\in [1, \infty]$ and let $w$ be a weight. The \emph{Sobolev space} $W^{m,p}(\R^d,w;X)$ is defined as the space of functions $f\in L^p(\R^d;w,X)$ for which $D^{\alpha} f$, taken in a distributional sense, is in $L^p(\R^d,w;X)$ for all multiindices $\alpha$ with $|\alpha|\leq m$. Let
\[\|f\|_{W^{m,p}(\R^d,w;X)} := \sum_{|\alpha|\leq m} \|D^{\alpha} f\|_{L^p(\R^d,w;X)}.\]
\end{definition}

For $s\in \R$ and $f\in \TD(\R^d;X)$, the Bessel potential  $J_{s} f\in \TD(\R^d;X)$ is defined by
\[J_{s} f := \F^{-1} [(1+|\cdot|^2)^{s/2} \wh{f}].\]
Obviously, $J_{s_1} J_{s_2} = J_{s_1 + s_2}$ for $s_1,s_2\in \R$, and $J_0$ is the identity mapping on $\TD(\R^d;X)$.

\begin{definition}
Let $X$ be a Banach space, $s\in \R$, $p\in [1, \infty)$ and let $w$ be a weight. Let the \emph{Bessel-potential space} $H^{s,p}(\R^d,w;X)$ be defined as the space of all $f\in \TD(\R^d;X)$ for which
$J_s f\in L^p(\R^d,w;X)$. Let
\[\|f\|_{H^{s,p}(\R^d,w;X)} := \|J_s f\|_{L^p(\R^d,w;X)}.\]
\end{definition}
It is immediate from the definition that
\begin{equation}\label{eq:JsigmaH}
 J_\sigma:H^{s,p}(\R^d,w;X) \to H^{s-\sigma,p}(\R^d,w;X) \quad \text{isomorphically}.
\end{equation}
Moreover, $W^{0,p}(\R^d,w;X) = H^{0,p}(\R^d,w;X) = L^p(\R^d,w;X)$. Certain
embeddings and identities between these spaces and
Triebel-Lizorkin spaces hold under geometric assumptions on $X$, see \eqref{H=F} and \eqref{H=W}.

\subsection{Density, lifting property, equivalent norms}

The elementary properties of the $A_\infty$-weighted spaces are the same as in the unweighted case.  Proposition \ref{prop:weightedmultiplier} allows to carry over the proofs of \cite[Section 2.3]{Tri83} and \cite[Section 15]{Tr97} to the weighted setting.

\begin{lemma}\label{density}
Let $X$ be a Banach space, $p\in (1,\infty)$, $q\in [1,\infty)$, $s\in \R$. Let $w\in A_{\infty}$.
The set $\Schw(\R^d;X)$ is dense in $B_{p,q}^s(\R^d,w;X)$, $F_{p,q}^s(\R^d,w;X)$ and $H^{s,p}(\R^d,w;X)$.
\end{lemma}
\begin{proof}
Let us consider $F_{p,q}^s(\R^d,w;X)$. Using Proposition \ref{prop:weightedmultiplier}, the same arguments as in Step 5 of the proof of \cite[Theorem 2.3.3]{Tri83} show that $f_N = \sum_{k=0}^N S_k f$ converges to $f$ in $F_{p,q}^s(\R^d,w;X)$ as $N\to \infty$. Still following \cite{Tri83}, let $\eta\in \Schw(\R^d)$ with $\eta(0) = 1$ and $\text{supp}(\hat{\eta}) \subset B(0,1)$.  Since $\F \eta(\delta\cdot) = \delta^{-d} \hat{\eta}(\delta^{-1}\cdot)$,  the support of $\F(\eta(\delta\cdot) f_N)$ is for all $\delta\in (0,1)$ contained in a ball that only depends on $N$. Applying again Proposition \ref{prop:weightedmultiplier}, we obtain that there is $C>0$, independent of $\delta$, such that
$$\|f_N-\eta(\delta\cdot) f_N\|_{F_{p,q}^s(\R^d,w;X)}  \leq C \|f_N-\eta(\delta\cdot) f_N\|_{L^p(\R^d,w;X)},$$
and the right-hand side tends to zero as $\delta\to 0$. Since $\eta(\delta\cdot) f_N \in \Schw(\R^d;X)$, the assertion for the Triebel-Lizorkin spaces follows. Similiar arguments show the assertion for the Besov
spaces.

For the density of $\Schw(\R^d)$ in $L^p(\R^d,w)$ see \cite[Exercise 9.4.1]{GraModern}. The invariance of $\Schw(\R^d;X)$ under the Bessel potential $J_s$ gives the density in $H^{s,p}(\R^d,w;X)$.\end{proof}

\begin{proposition}\label{prop:lifting}
Let $X$ be a Banach space, $p\in [1, \infty)$, $q\in [1, \infty]$, $s\in \R$  and $w\in A_{\infty}$. Then for all $\sigma \in \R$,
\begin{equation} \label{J-iso-B}
J_{\sigma}:B^{s}_{p,q}(\R^d,w;X)\to B^{s-\sigma}_{p,q}(\R^d,w;X) \qquad \text{isomorphically},
\end{equation}
\begin{equation}\label{J-iso-F}
 J_{\sigma}:F^{s}_{p,q}(\R^d,w;X)\to F^{s-\sigma}_{p,q}(\R^d,w;X) \qquad \text{isomorphically}.
\end{equation}

\end{proposition}
\begin{proof}Choose $r\in (0,\min\{p,q\})$ such that $w\in A_{p/r}$. Using Proposition \ref{prop:weightedmultiplier}, the same proof as in the unweighted case gives the assertions  (see \cite[Theorem 2.3.8]{Tri83}).
\end{proof}

\begin{proposition}\label{prop:differentiation}
Let $X$ be a Banach space, $p\in [1, \infty)$, $q\in [1, \infty]$, $s\in \R$ and $w\in A_{\infty}$. Then for all $k\in \N$ it holds that
\begin{equation}\label{eq:Besovdifferentiation}
\sum_{|\alpha|\leq k} \|D^\alpha f\|_{B_{p,q}^{s-k}(\R^d,w;X)} \quad \text{defines an equivalent norm on }\, B_{p,q}^{s}(\R^d,w;X);
\end{equation}
\begin{equation}\label{eq:TLdifferentiation}
\sum_{|\alpha|\leq k} \|D^\alpha f\|_{F_{p,q}^{s-k}(\R^d,w;X)} \quad \text{defines an equivalent norm on }\, F_{p,q}^{s}(\R^d,w;X).
\end{equation}
\end{proposition}
\begin{proof}In the unweighted scalar case, these results are shown in \cite[Theorem 2.3.8]{Tri83}. The proofs are essentially based on a multiplier theorem of Mihlin-H\"ormander type in Besov and Triebel-Lizorkin spaces, see \cite[Theorem 2.3.7]{Tri83} for the scalar and \cite[Section 15.6]{Tr97} for the vector-valued case. Using Proposition \ref{prop:weightedmultiplier}, the proof given in \cite[Section 15.6]{Tr97} carries over to the weighted setting, for all $w\in A_\infty$.\end{proof}

\subsection{Elementary embeddings}

The elementary embedding properties and their proofs for the above vector-valued function spaces are the same as in the unweighted case (see \cite[Section 2.3.2]{Tri83}):

\begin{proposition}\label{prop:elementaryembeddingBTL}
Let $X$ be a Banach space and $w\in A_{\infty}$.
\begin{enumerate}[(1)]
\item For all $1\leq q_0\leq q_1\leq \infty$ and $s\in \R$ one has
\[B_{p,q_0}^s (\R^d,w;X)\hookrightarrow B_{p,q_1}^s (\R^d,w;X), \qquad p\in [1, \infty],\]
\[F_{p,q_0}^s (\R^d,w;X)\hookrightarrow F_{p,q_1}^s (\R^d,w;X),  \qquad p\in [1, \infty).\]

\item For all $q_0,q_1\in [1, \infty]$, $s\in \R$ and $\varepsilon>0$ one has
\[B_{p,q_0}^{s+\varepsilon} (\R^d,w;X)\hookrightarrow B_{p,q_1}^s (\R^d,w;X), \qquad p\in [1, \infty],\]
\[F_{p,q_0}^{s+\varepsilon} (\R^d,w;X)\hookrightarrow F_{p,q_1}^s (\R^d,w;X), \qquad p\in [1, \infty).\]

\item For all $q\in [1, \infty]$, $s\in \R$ and $p\in [1, \infty)$ one has
\[B_{p,\min\{p,q\}}^{s} (\R^d,w;X)\hookrightarrow F_{p,q}^{s} (\R^d,w;X) \hookrightarrow B_{p,\max\{p,q\}}^{s} (\R^d,w;X).\]

\end{enumerate}
\end{proposition}

The $H$-spaces are related to the $B$- and $F$-spaces as follows.

\begin{proposition}\label{prop:squeezeHF}
Let $X$ be a Banach space, $s\in \R$, $p\in (1, \infty)$ and $w\in A_p$. Then
\begin{align}
\label{eq:BesovH}
B^{s}_{p,1}(\R^d,w;X) & \hookrightarrow H^{s,p}(\R^d,w;X) \hookrightarrow B^{s}_{p,\infty}(\R^d,w;X),
\\ \label{eq:TriebelLizorkinH}
F^{s}_{p,1}(\R^d,w;X) & \hookrightarrow H^{s,p}(\R^d,w;X) \hookrightarrow F^{s}_{p,\infty}(\R^d,w;X).
\end{align}
Moreover, if $m\in \N$, then
\begin{align}
\label{eq:BesovW} B^{m}_{p,1}(\R^d,w;X) & \hookrightarrow W^{m,p}(\R^d,w;X) \hookrightarrow B^{m}_{p,\infty}(\R^d,w;X),
\\ \label{eq:TriebelLizorkinW}
F^{m}_{p,1}(\R^d,w;X) & \hookrightarrow W^{m,p}(\R^d,w;X) \hookrightarrow F^{m}_{p,\infty}(\R^d,w;X).
\end{align}
\end{proposition}
\begin{proof} Using Propositions \ref{prop:lifting} and \ref{prop:differentiation}, it suffices to consider the cases $s=m=0$. But then (\ref{eq:BesovW}) and (\ref{eq:TriebelLizorkinW}) are the same as (\ref{eq:BesovH}) and (\ref{eq:TriebelLizorkinH}). Further, \eqref{eq:BesovH} follows from \eqref{eq:TriebelLizorkinH} and Proposition \ref{prop:elementaryembeddingBTL}. To prove \eqref{eq:TriebelLizorkinH}, we extend the argument in \cite[Proposition 2]{SchmSi05}.

For $f\in \Schw(\R^d;X)$ we have $f = \sum_{n\geq 0} S_n f$ almost everywhere on $\R^d$. Therefore
\[\|f\|_{X} \leq \sum_{n\geq 0}\|S_n f\|_{X},\]
and the first part of \eqref{eq:TriebelLizorkinH} follows by taking $L^p(\R^d,w)$-norms and Lemma \ref{density}.

Next let $f\in L^p(\R^d,w;X)$. By the definition of $S_n$ and $\varphi_n$ for each $x\in \R^d$, one has
\begin{align*}
\|S_n f(x) \|_X & \leq \int_{\R^d} |\varphi_n(y)| \|f(x-y)\|_X \, dy\leq 2\sup_{n\geq 0}\int_{\R^d} |2^{n d}\varphi_0(2^n y)| \|f(x-y)\|_X \, dy.
\end{align*}
Let $\psi:\R^d\to \R_+$ be defined by $\psi(y) = \sup\{\varphi_0(z): |z|\geq |y|\}$.
It follows from Lemma \ref{lem:Stein} that
\begin{align*}
\|f\|_{F^0_{p,\infty}(\R^d,w;X)} &= \Big\|\sup_{n\geq 0}\|S_n f \|_X\Big\|_{L^p(\R^d,w)}
 \leq 2 \|\psi\|_{L^1(\R^d)} C_{p,w} \|f\|_{L^p(\R^d,w;X)},
\end{align*}
which proves the second part of \eqref{eq:TriebelLizorkinH}.\end{proof}

\begin{remark}\label{remarkProp3.12}
 The proof shows that the embeddings on the left-hand sides in the above proposition are also true for $p=1$ and $w\in A_{\infty}$. For the embeddings on the right-hand side this is different:

For $p\in (1,\infty)$, let $A_p^{{\rm loc}}$ be the class of weights defined in \cite{Ry01}, i.e., $w\in A_p^{{\rm loc}}$ if
$$\sup_{|Q|\leq 1} \frac{1}{|Q|^p} \left (\int_Q w d x\right)\left (\int_Q w' dx \right)^{p-1} <\infty,$$
where the supremum is taken over all cubes $Q$ in $\R^d$ with sides parallel to the coordinate axes.  For weights of the form $w(x) = |x|^{\gamma}$ one can check that $w\in A_p^{{\rm loc}}$ if and only if $w\in A_p$. Now we claim that
\begin{equation}\label{eq:LpF0}
L^p(\R^d,w) \hookrightarrow F^{0}_{p,\infty}(\R^d,w)
\end{equation}
if and only if $w\in A_p^{{\rm loc}}$.
\end{remark}
\begin{proof} If $w\in A_p^{{\rm loc}}$, then \eqref{eq:LpF0} follows from $L^p(\R^d,w) = F^{0}_{p,2}(\R^d,w)$ (see \cite[p. 178]{Ry01}) and the embedding $\ell^2\hookrightarrow \ell^\infty$.

Conversely, assume that \eqref{eq:LpF0} holds for a weight $w$. Let $(\varphi_j)_{j\geq 0}\in \Phi$. Using the continuity of $\varphi_1$ and $\varphi_1(0)>0$, we can find $c>0$ and $N\in \N$ such that
$\text{Re} (\varphi(x))\geq  c$ for all $|x|\leq d 2^{-N+2}$. Let $Q$ be a cube with $|Q|\leq 2^{-Nd}$. Let $f:\R^d\to \R$ be a function which satisfies $f\geq 0$ on $Q$ and $f = 0$ on $\R^d\setminus Q$. Let $j\in \N$ be such that $2^{-(j+N)d} \leq |Q|< 2^{-(j+N-1)d}$. Denoting by $\ell(Q)$ the maximal axis length of $Q$, it then holds $\ell(Q) <2^{-(j+N-1)}$. Now for every $x\in Q$ one has
\begin{align*}
|\varphi_{j+1}*f(x)| = 2^{jd} \Big|\int_Q  \varphi_1(2^j(x-y)) f(y) \, dy\Big|  \geq 2^{jd} c \int_Q f(y) \, dy \geq \frac{C}{|Q|} \int_Q f(y) \, dy,
\end{align*}
where we used that
\[|2^j (x- y)|\leq  2^j 2 d \ell(Q)< d 2^{-N+2} \qquad \text{for all }\,x,y\in Q.\]
Let $\lambda = \frac{C}{|Q|} \int_Q f(y) \, dy$. The above estimate and \eqref{eq:LpF0} yield
\begin{align*}
\int_Q w(x) \, dx \leq \lambda^{-p} \int_Q |\varphi_{j+1} * f(x)|^p w(x) \, dx \leq C \lambda^{-p} \|f\|_{L^p(\R^d,w)}^p.
\end{align*}
Rewriting this gives
\[\Big(\int_Q w(x) \, dx\Big) \Big(\int_Q f(y) \, dy\Big)^p \leq C |Q|^p  \int_Q |f(x)|^p w(x)\, dx.\]
As in \cite[Equation (3.12)]{Kurtz} (basically by taking $f = \one_Q w'$) this implies
\[\Big(\frac{1}{|Q|}\int_Q w(x) \, dx\Big) \Big(\frac{1}{|Q|}\int_Q w' \, dy\Big)^{p-1}\leq C.\]
Since the definition of $A_p^{{\rm loc}}$ is independent of the upper bound for the cube size (see \cite[Remark 1.5]{Ry01}) we obtain that $w\in A_p^{{\rm loc}}$.
\end{proof}

\section{Embeddings for Besov Spaces -- Proof of Theorem \ref{thm:main1}\label{sec:Besov}}

From now on we specialize to power weights, i.e., $w(x) = |x|^\gamma$ with $\gamma>-d$. We first consider sufficient conditions for the embedding, and show their optimality in the next subsection.

\subsection{Sufficient conditions} In this subsection we prove the sufficiency part for the embedding for Besov spaces, i.e. Theorem \ref{thm:main1}  (2) $\Rightarrow$ (1). The main ingredient of the proof is the following two-weight version of an inequality of Plancherel-Polya-Nikol'skij type. As already mentioned in the introduction, a completely different proof for the scalar version of Theorem \ref{thm:main1} is given in \cite{HaSkr}.
\begin{proposition}\label{prop:Nikolskii}
Let $X$ be a Banach space and let $1<p_0,p_1\leq \infty$. Let $\gamma_0,\gamma_1>-d$ and $w_0(x) = |x|^{\gamma_0}$ and $w_1(x) = |x|^{\gamma_1}$.  Suppose
\begin{equation}\label{cond:nikolskii}
 \frac{\gamma_1}{p_1}\leq \frac{\gamma_0}{p_0} \qquad \text{and} \qquad \frac{d+\gamma_1}{p_1}<\frac{d+\gamma_0}{p_0}.
\end{equation}
Let $f:\R^d\to X$ be a function with $\text{supp}(\widehat{f}) \subseteq \{x\in \R^d: |x|<t\}$, where $t>0$ is fixed.
Then for any multiindex $\alpha$ there is a constant $C_\alpha$, independent of $f$ and $t$, such that
\begin{equation}\label{eq:twoweightnikolskii}
\|D^\alpha f\|_{L^{p_1}(\R^d,w_1;X)}\leq C_\alpha t^{|\alpha| + \delta} \|f\|_{L^{p_0}(\R^d,w_0;X)},
\end{equation}
where $\delta = \frac{d+ \gamma_0}{p_0} - \frac{d + \gamma_1}{p_1}> 0$.
\end{proposition}
\begin{remark}\
\begin{enumerate}[(i)]
\item Suppose that \eqref{eq:twoweightnikolskii} holds true for $\alpha =0$ and all $f$ as in the proposition. Then it follows from the proof of Theorem \ref{thm:main1} given below that embedding \eqref{eq:introBesov} for Besov spaces holds true (with suitable chosen $s_0,s_1$). By the necessary conditions for this embedding obtained in the  Propositions \ref{prop:necc} and \ref{prop:neccnonopt2} below, we conclude that \eqref{eq:twoweightnikolskii} holds true if and only if either $p_0 = p_1$ and $\gamma_0 = \gamma_1$ or \eqref{cond:nikolskii} are satisfied.
\item The case where $w_0=w_1$ is an $A_\infty$ weight which satisfies $\inf_{x\in \R^d} w(B(x,t))\geq t^\varepsilon$ with $\varepsilon>0$ is considered in \cite[Lemma 2.5]{Bui82}. A part of the argument in \cite[Lemma 2.5]{Bui82} will be repeated in \eqref{eq:fLinftyBui} below, because the details are  needed again at a later point.
\item It would be interesting to find a two-weight characterization for \eqref{eq:twoweightnikolskii} for general weights $w_0$ and $w_1$ in case $t=1$ or more general $t$. There might be a connection to \cite[Proposition 2.1]{HaSkr}.
\item Certain weighted version of inequalities of Nikol'skij type can also be found in \cite[Sections 1.3.4 and 6.2.3]{Tri83}. However, the power weights we consider are not covered by those results.
\item It follows from the proof below that for $\gamma_0, \gamma_1\geq 0$, Proposition \ref{prop:Nikolskii} holds for all $p_0, p_1\in (0,\infty)$ which satisfy \eqref{cond:nikolskii}.
\end{enumerate}
\end{remark}

For the proof of the proposition we  make use of the following weighted version of Young's inequality (see \cite[Theorem 3.4 (3.7)]{Kerman} or \cite[Theorem 2.2 (ii)]{Bui94}). The proof is based on the Stein--Weiss result on fractional integration
(see \cite{Beck08} for a short proof).
\begin{lemma}\label{lem:Young} Let $1 < q \leq r < \infty$ and $a,b,c\in \R$ be such that
\begin{align}
\label{as:y1} \frac1r  =  \frac1q - 1 + \frac{a+b+c}{d}, \qquad b+c \geq 0, \qquad 0< a < d, \qquad b<d\big(1-\tfrac1q\big), \qquad c < \frac{d}{r}.
\end{align}
Then there is a constant $C$ such that for all measurable functions $f$ and $g$ one has
\begin{equation}\label{eq:YoungLrLinftyLq}
\|x\mapsto f*g(x) |x|^{-c}\|_{L^r(\R^d)}\leq C \|x\mapsto f(x) |x|^{a}\|_{L^\infty(\R^d)} \|x\mapsto g(x) |x|^{b}\|_{L^q(\R^d)}.
\end{equation}
\end{lemma}

\begin{remark}\label{Young-nec}
In \cite[Theorem 2.1]{Bui94} several necessary conditions for weighted Young's inequalities are obtained. In our situation we obtain another necessary condition which also appears in the sufficient conditions for \eqref{eq:YoungLrLinftyLq} in \cite[Condition (14)]{Bui94}. In fact, it follows from the proof below that if \eqref{eq:YoungLrLinftyLq} holds for some $b$ and $c$ with $b+c<0$, then one obtains \eqref{eq:introBesov} with $\gamma_1/p_1>\gamma_0/p_0$. This is impossible according to Theorem \ref{thm:main1}. Therefore, $b+c\geq 0$ from \eqref{as:y1} is also necessary for \eqref{eq:YoungLrLinftyLq}. With some additional arguments one can derive the same necessary condition if the $L^\infty$-norm in \eqref{eq:YoungLrLinftyLq} is replaced by an $L^r$-norm. In a similar way one can see that $a<d$ from \eqref{as:y1} is necessary for \eqref{eq:YoungLrLinftyLq}.
\end{remark}

The next lemma is stated without proof in \cite[Condition $B_p$]{Bui82}. We include a proof for convenience.
\begin{lemma}\label{lem:Tx}
Let $p\in (1, \infty)$ and $w\in A_p$. Then there is a constant $C$ such that for all $x\in \R^d$ one has
\[\int_{\R^d}  w(y)  (1+|x-y|)^{-d p} \, dy \leq C\int_{B(x,1)} w(y) \, dy.\]
\end{lemma}

\begin{proof}
If we let  $g_x(z) = \one_{B(x,1)}(z)$ and $s_0 = |x-y|+1$, then its maximal function $Mg_x$ satisfies
\begin{align*}
(M g_x)(y) & = \sup_{s>0} \frac{|B(y,s)\cap B(x,1)| }{|B(y,s)|}
\geq \frac{|B(y,s_0)\cap B(x,1)| }{|B(y,s_0)|}  = \frac{|B(x,1)| }{|B(y,s_0)|}= s_0^{-d} = (|x-y| + 1)^{-d}.
\end{align*}
Therefore, since $w\in A_{p}$, we see that
\[
\int_{\R^d}  w(y) (1+|x-y|)^{-dp} \, dy \leq \|M g_x\|_{L^{p}(\R^d,w)}^{p} \leq C\|g_x\|_{L^{p}(\R^d,w)}^{p} = C\int_{B(x,1)} w(y) \, dy,
\]
and the proof is finished.
\end{proof}

\begin{proof}[Proof of Proposition \ref{prop:Nikolskii}]
By a scaling argument it suffices to consider the case $t=1$. Let $B_1= \{x\in \R^d:|x|<1\}$.
For $\alpha \neq 0$, the same arguments as in the proof of \cite[Proposition 1.3.2]{Tri83} show that
$$\|D^\alpha f\|_{L^{p_1}(\R^d,w_1;X)} \leq C_\alpha \|f\|_{L^{p_1}(\R^d,w_1;X)}$$
for all $p_1\in (1,\infty]$ and $\gamma_1 > - d$. We may thus restrict to the case $\alpha = 0$.

The proof of \eqref{eq:twoweightnikolskii} for $\alpha = 0$ is split into several cases of which some are overlapping. In Case 1 we treat $p_1 = \infty$, in Cases 2-6 we consider $p_0\leq p_1<\infty$, and in Case 7  we derive the estimate for $p_0>p_1$.

{\em Case 1:} $p_0<\infty$ and $p_1 = \infty$. Then $\gamma_0 \geq 0$ by assumption,  and \eqref{eq:twoweightnikolskii} is independent of $\gamma_1$ due to $L^\infty(\R^d,w_1;X) = L^\infty(\R^d;X)$. We follow the arguments given in \cite[Lemma 2.5]{Bui82}.
First assume that $f\in L^\infty(\R^d;X)$. Let $\eta\in \Schw(\R^d)$ be such that $\text{supp}(\widehat{\eta})\subseteq B_2$ and $\widehat{\eta} = 1$ on $B_1$. Then one has $f = f*\eta$. Let $q\in (0,1)$ be so small that $\frac{\gamma_0}{r-1}<d$, where $r = \frac{p_0}{q}$. Then $w_0' = w_0^{-1/(r-1)}\in A_{r'}$.
For any $x\in \R^d$, H\"older's inequality with $\frac{1}{r} + \frac{1}{r'} = 1$ yields
\begin{equation}\label{eq:fLinftyBui}\begin{aligned}
\|f(x)\|& \leq \int_{\R^d} \|f(y)\| |\eta(x-y)| \, dy
\\ & \leq \|f\|^{1-q}_{L^\infty(\R^d;X)}  \int_{\R^d}  (\|f(y)\| |y|^{\gamma_0/p_0})^q |y|^{-q\gamma_0/p_0}|\eta(x-y)| \, dy
\\ & \leq \|f\|^{1-q}_{L^\infty(\R^d;X)}  \|f\|_{L^{p_0}(\R^d,w_0;X)}^{q}  \Big(\int_{\R^d}  w_0'(y)  |\eta(x-y)|^{r'} \, dy\Big)^{1/r'}.
\end{aligned}\end{equation}
Now since $\eta$ is a Schwartz function, there is a constant $C$ such that $|\eta(y)|^{r'}\leq C(1+|y|)^{-dr'}$ for all $y\in \R^d$.
Therefore, it follows from Lemma \ref{lem:Tx} and $-\frac{\gamma}{r-1} > -d$ that for all $x\in \R^d$ we have
\[
\Big(\int_{\R^d}  w_0'(y)  |\eta(x-y)|^{r'} \, dy\Big)^{1/r'}\leq C \int_{B(x,1)} w_0'(y)\, dy = C \int_{B(x,1)} |y|^{-\gamma_0/(r-1)} d y \leq C_{p_0,\gamma_0,d}.
\]
Combining this with \eqref{eq:fLinftyBui} we obtain
\begin{equation}\label{eq:estbuiLinfty2}
\|f\|_{L^\infty(\R^d;X)} \leq C  \|f\|_{L^{p_0}(\R^d,w_0;X)}.
\end{equation}

If $f\notin L^\infty(\R^d;X)$, then as in \cite{Bui82} we take a function $\phi\in \Schw(\R^d)$ with $\phi(0) = 1$ and $\supp \wh{\phi} \subset B_1$. Since $f$ is smooth, we have that $\phi(r\cdot)  f \in L^\infty(\R^d;X)$ for all $r>0$. For $x\in \R^d$, embedding \eqref{eq:estbuiLinfty2} the dominated convergence theorem imply that
$$\|f(x)\|  = \lim_{r\to 0} \|\phi(rx)  f(x)\| \leq C\lim_{r\to 0} \|\phi(r\cdot) f\|_{L^{p_0}(\R^d,w_0;X)} = C\|f\|_{L^{p_0}(\R^d,w_0;X)}.$$

{\em Case 2:} $p_0\leq p_1<\infty$ and $\gamma:=\gamma_0 =\gamma_1 \geq 0$. Let $w(x) = |x|^{\gamma}$. We can assume $p_0<p_1$, since the other case is trivial. By Case 1 we know that $f\in L^\infty(\R^d;X)$, and from \eqref{eq:estbuiLinfty2} we obtain
\begin{align*}
\|f\|_{L^{p_1}(\R^d,w;X)} \leq \|f\|_{L^\infty(\R^d;X)}^{1-\frac{p_0}{p_1}} \|f\|_{L^{p_0}(\R^d, w;X)}^{p_0/p_1}
\leq C \|f\|_{L^{p_0}(\R^d, w;X)}.
\end{align*}

{\em Case 3:} $p_0\leq p_1<\infty$ and $0\leq \gamma_1< \gamma_0$. First consider the integral over $\R^d\setminus B_1$. Since $|x|^{\gamma_1} \leq |x|^{\gamma_0}$ for $|x|\geq 1$, one has
\begin{align*}
\|f\|_{L^{p_1}(\R^d\backslash B_1,w_1;X)}& \leq \|f\|_{L^{p_1}(\R^d\backslash B_1,w_0;X)}
 \leq \|f\|_{L^{p_1}(\R^d,w_0;X)} \leq C\|f\|_{L^{p_0}(\R^d,w_0;X)},
\end{align*}
where in the last step we applied Case 2. For the integral over $B_1$, let $p = p_1 \frac{\gamma_0}{\gamma_1} > p_1$. We apply H\"older's inequality with  $\frac{\gamma_1}{\gamma_0} + \frac{1}{r}=1$, to obtain
\begin{align*}
\|f\|_{L^{p_1}(B_1,w_1;X)}& \leq \|f\|_{L^{p}(\R^d,w_0;X)} |B_1|^{1/p_1r}
\leq C\|f\|_{L^{p_0}(\R^d,w_0;X)},
\end{align*}
where in the last step we used Case 2.

{\em Case 4:} $p_0\leq p_1<\infty$ and $0\leq \gamma_0<\gamma_1$. Then necessarily $p_0 < p_1$.
First consider the integral over $B_1$. Since $\gamma_1 > \gamma_0$, we have
\begin{align*}
\|f\|_{L^{p_1}(B_1,w_1;X)} & \leq \|f\|_{L^{p_1}(B_1,w_0;X)} \leq \|f\|_{L^{p_1}(\R^d,w_0;X)}
\leq C\|f\|_{L^{p_0}(\R^d,w_0;X)},
\end{align*}
where in the last step we applied Case 2. Next consider $\R^d \backslash B_1$. We have
\begin{equation}\label{eq:superestweight}
\begin{aligned}
 \|f\|_{L^{p_1}(\R^d\backslash B_1,w_1;X)} &\, = \Big ( \int_{\R^d \backslash B_1} \|f(x)\|^{p_0} |x|^{\gamma_0} \|f(x)\|^{p_1-p_0} |x|^{\gamma_1-\gamma_0} \, dx \Big)^{1/p_1}\\
&\, \leq \|f\|_{L^{p_0}(\R^d,w_0;X)}^{p_0/p_1} \Big ( \sup_{x\in \R^d \backslash B_1} \|f(x)\| |x|^{(\gamma_1-\gamma_0)/(p_1-p_0)}\Big)^{1-p_0/p_1}.
\end{aligned}\end{equation}
Now fix $x\in \R^d$ with $|x|\geq 1$. Since $f\in L^\infty(\R^d;X)$,  inequality \eqref{eq:fLinftyBui} implies that
\[\|f(x)\| \leq C \|f\|_{L^{p_0}(\R^d,w_0;X)} \Big(\int_{\R^d}  w'(y)  |\eta(x-y)|^{r'} \, dy\Big)^{\frac{1}{r'q}},\]
where $r = \frac{p_0}{q}$ as in Case 1 and $w'(y) = |y|^{-\gamma_0/(r-1)}$. By Lemma \ref{lem:Tx}, for $|x|\geq 1$ one can estimate
\[\int_{\R^d}  w'(y)  |\eta(x-y)|^{r'} \, dy \leq C\int_{B(x,1)} w'(y) \, dy\leq C|x|^{-\gamma_0/(r-1)},\]
and thus $(r-1)r' q = p_0$ yields
\begin{equation}\label{eq:weightedfest}
\|f(x)\| \leq C \|f\|_{L^{p_0}(\R^d,w_0;X)} |x|^{-\gamma_0/p_0},  \qquad  x\in \R^d\setminus B_1.
\end{equation}
Substituting \eqref{eq:weightedfest} into \eqref{eq:superestweight}, we obtain
\[\|f\|_{L^{p_1}(\R^d\backslash B_1,w_1;X)} \leq C \|f\|_{L^{p_0}(\R^d,w_0;X)} \sup_{x\in \R^d \backslash B_1} |x|^{-\gamma_0/p_0 + (\gamma_1-\gamma_0)/(p_1-p_0)}\leq C \|f\|_{L^{p_0}(\R^d,w_0;X)},\]
where the last estimate is a consequence of $|x|\geq 1$ and $\gamma_1/p_1 \leq \gamma_0/p_0$.

{\em Case 5:} $p_0 \leq p_1<\infty$, $-d<\gamma_0<d (p_0-1)$ and $-d<\gamma_1$. Then $\gamma_1 < d(p_1-1)$ by assumption. Let $\eta\in \Schw(\R^d)$ be as in Case 1. For all $x\in \R^d$ we then have
\begin{align*}
\|f(x)\|\leq g*|\eta|(x)
\end{align*}
where $g(x) = \|f(x)\|$. Set
$$
r=p_1, \qquad q = p_0, \qquad  a = d- \frac{d+\gamma_0}{p_0} + \frac{d+\gamma_1}{p_1}, \qquad   b = \frac{\gamma_0}{p_0}, \qquad  c = -\frac{\gamma_1}{p_1}.
$$
We claim the conditions of Lemma \ref{lem:Young} hold. Indeed, note that
$b+c \geq 0$ is equivalent to $\frac{\gamma_1}{p_1} \leq \frac{\gamma_0}{p_0}$, $a>0$ follows from $-\gamma_0>-d(p_0-1)$ and $\gamma_1 > -d$, and $a < d$ follows from  $-\frac{d+\gamma_0}{p_0} + \frac{d+\gamma_0}{p_0} <0$.  The estimate $b<d(1-\tfrac1q)$ is equivalent to $\gamma_0<d(p_0-1)$, and $c < \frac{d}{r}$ is equivalent to $\gamma_1 > -d$. Therefore, Lemma \ref{lem:Young} yields
\begin{align*}
 \|f\|_{L^{p_1}(\R^d,w_1;X)} &\, \leq \|g*|\eta|\|_{L^{p_1}(\R^d,w_1)} \\
&\, \leq C\|x\mapsto |x|^{a} \eta(x)\|_{L^\infty(\R^d)} \|g\|_{L^{p_0}(\R^d,w_0)} = C \|f\|_{L^{p_0}(\R^d,w_0;X)}.
\end{align*}

{\em Case 6:} $p_0 \leq p_1<\infty$, $d(p_0-1)\leq \gamma_0$ and $-d < \gamma_1<0$.
Let $\gamma\in (0,d(p_0-1))$ be arbitrary, and set $w(x) = |x|^{\gamma}$. Then $\gamma_1/p_1 \leq 0 <\gamma/p_0<\gamma_0/p_0$. Therefore, by Cases 5 and 3 we have
\[\|f\|_{L^{p_1}(\R^d,w_1;X)}\leq C \|f\|_{L^{p_0}(\R^d,w;X)}\leq C\|f\|_{L^{p_0}(\R^d,w_0;X)}.\]

{\em Case 7:} $p_1< p_0 <\infty$ and $\gamma_0,\gamma_1 > -d$. For $\varepsilon>0$ we use H\"older's inequality with respect
to the finite measure $|x|^{-d-\varepsilon} d x$ on $\R^d\backslash B_1$, to
obtain
\begin{align*}
\|f\|_{L^{p_1}(\R^d\backslash B_1,w_1)} &\,= \Big (\int_{\R^d \backslash
B_1} \|f(x)\|^{p_1} |x|^{\gamma_1 +d
+\varepsilon} |x|^{-d-\varepsilon} d x \Big)^{1/p_1}\\
&\,\leq C  \Big (\int_{\R^d \backslash B_1} \|f(x)\|^{p_0}
|x|^{\frac{p_0}{p_1}(\gamma_1 +d +\varepsilon) -d-\varepsilon} d x
\Big)^{1/p_0}.
\end{align*}
The assumption $\frac{\gamma_1+d}{p_1}< \frac{\gamma_0+d}{p_0}$ implies that
$\frac{p_0}{p_1}(\gamma_1 +d +\varepsilon) -d-\varepsilon \leq \gamma_0$ for
sufficiently small $\varepsilon$. Employing this, it follows that
\[\|f\|_{L^{p_1}(\R^d\backslash B_1,w_1)} \leq C \|f\|_{L^{p_0}(\R^d,w_0)}.\]
To estimate the weighted $L^{p_1}$-norms over $B_1$, for $\varepsilon>0$ we use H\"older's inequality with respect to the
finite measure $|x|^{-d +\varepsilon} d x$ on $B_1$, which gives
\begin{align*}
\|f\|_{L^{p_1}( B_1,w_1;X)} &\,= \left (\int_{B_1} \|f(x)\|^{p_1}
|x|^{\gamma_1 +d -\varepsilon}
|x|^{-d+\varepsilon} d x \right)^{1/p_1}\\
&\,\leq C  \left (\int_{B_1} \|f(x)\|^{p_0}
|x|^{\widetilde{\gamma}_0} d x \right)^{1/p_0} \leq C
\|f\|_{L^{p_0}(\R^d,\widetilde{w}_0;X)}.
\end{align*}
Here in the last line we have set $\widetilde{w}_0(x) =
|x|^{\widetilde{\gamma}_0}$ with $\widetilde{\gamma}_0 :=
\frac{p_0}{p_1}(\gamma_1 +d -\varepsilon) -d+\varepsilon$.
Observe that if $\varepsilon$ is sufficiently small, then the assumptions
$\gamma_1+ d>0$ and $\frac{d +\gamma_1}{p_1} < \frac{ d + \gamma_0}{p_0}$ imply
$-d <\widetilde{\gamma}_0$ and $\frac{\widetilde{\gamma}_0}{p_0} <
\frac{\gamma_0}{p_0},$ respectively.
Therefore, by Cases 2-6 we obtain that
\[\|f\|_{L^{p_0}(\R^d,\widetilde{w}_0;X)} \leq \|f\|_{L^{p_0}(\R^d,w_0;X)}.\]
Combing these estimates yields \eqref{eq:twoweightnikolskii}.\end{proof}

\begin{proof}[Proof of Theorem \ref{thm:main1} (2) $\Rightarrow$ (1)] If  \eqref{cond:trivial} holds, then embedding (1) follows from Proposition \ref{prop:elementaryembeddingBTL}. Moreover, if \eqref{cond:expnew} implies (1), then (1) also follows from \eqref{cond:exp} by Proposition \ref{prop:elementaryembeddingBTL}.

Assume that \eqref{cond:expnew} holds. For (1) it suffices to consider the case $q:=q_0 = q_1$. Let $\delta = \frac{\gamma_0+d}{p_0} - \frac{\gamma_1+d}{p_1}$. Since $\widehat{S_n f} = \widehat{\varphi}_n \widehat{f}$ is supported in $\{x\in \R^d: |x|<3 \cdot 2^{n-1}\}$, Proposition \ref{prop:Nikolskii} gives
\[\|S_nf\|_{L^{p_1}(\R^d,w_1;X)}\leq C 2^{\delta n} \|S_nf\|_{L^{p_0}(\R^d,w_0;X)},  \qquad n\geq 0.\]
Therefore, using $s_1 + \delta = s_0$, we find that
\begin{align*}
\|f\|_{B^{s_1}_{p_1, q}(\R^d,w_1;X)} & = \big\|(2^{s_1 n}\|S_nf\|_{L^{p_1}(\R^d,w_1;X)})_{n\geq 0}\big\|_{\ell^{q}}
\\ & \leq C \big\|(2^{s_0 n}\|S_nf\|_{L^{p_0}(\R^d,w_0;X)})_{n\geq 0}\big\|_{\ell^{q}} = C \|f\|_{B^{s_0}_{p_0, q}(\R^d,w_0;X)}.
\end{align*}
\end{proof}

\subsection{Necessary conditions}

In this subsection we prove Theorem \ref{thm:main1} (1) $\Rightarrow$ (2). We start with an elementary lemma.

\begin{lemma}\label{lem:phinphij}
Let $p\in [1, \infty]$ and $w(x) = |x|^\gamma$ with $\gamma>-d$. Let $(\varphi_n)_{n \geq 0}\in \Phi$ and $j\in \{-1,0,1\}$. Then there is a constant $C_{\varphi,p,\gamma,j}$ such that for every $n\geq 2$ one has
\[\|\varphi_n*\varphi_{n+j}\|_{L^{p}(\R^d,w)} = C_{\varphi,p,\gamma,j} 2^{nd} 2^{-n\frac{d+\gamma}{p}}.\]
\end{lemma}
\begin{proof} Since $\wh{\varphi}_n =  \wh{\varphi}_1(2^{-n+1}\cdot)$, this follows from a straightforward substitution argument.
\end{proof}

We give necessary conditions for an embedding for general $p_0$ and $p_1$.

\begin{proposition}\label{prop:necc}
Let $1< p_0,p_1 \leq \infty$  and $s_0,s_1\in \R$. Let further $w_0(x) = |x|^{\gamma_0}$ and $w_1(x) = |x|^{\gamma_1}$ with $\gamma_0,\gamma_1>-d$. Suppose
\begin{equation}\label{nec-embedd}
 B^{s_0}_{p_0,1}(\R^d,w_0) \hookrightarrow B^{s_1}_{p_1, \infty}(\R^d,w_1).
\end{equation}
Then
\begin{equation}\label{eq:necparameters}
s_0 - \frac{d+\gamma_0}{p_0} \geq s_1 - \frac{d+\gamma_1}{p_1}, \qquad
\frac{\gamma_1+d}{p_1}\leq \frac{\gamma_0+d}{p_0} \qquad  \text{and} \qquad
\frac{\gamma_1}{p_1}\leq \frac{\gamma_0}{p_0}.
\end{equation}
\end{proposition}
\begin{proof}
Let $(\varphi_n)_{n \geq 0}\in \Phi$. By \eqref{nec-embedd}, for every $n\geq 0$ one has
\[\|\varphi_n\|_{B^{s_1}_{p_1, \infty}(\R^d,w_1)} \leq C \|\varphi_n\|_{B^{s_0}_{p_0, 1}(\R^d,w_0)}.\]
By the assumption on the support of $(\hat{\varphi}_n)_{n\geq 0}$ one gets that
\[2^{n s_1}\|\varphi_n*\varphi_n\|_{L^{p_1}(\R^d,w_1)}\leq C 2^{n s_0} \sum_{j=-1}^{1} \|\varphi_n*\varphi_{n+j}\|_{L^{p_0}(\R^d,w_0)}, \qquad n\geq 0.\]
Using Lemma \ref{lem:phinphij}, this implies that there is a constant $\tilde{C}$ such that
\[2^{n s_1} 2^{nd} 2^{-n\frac{d+\gamma_1}{p_1}} \leq \tilde{C} 2^{n s_0} 2^{nd} 2^{-n\frac{d+\gamma_0}{p_0}}.\]
Letting $n$ tend to infinity gives $s_1 - \frac{d+\gamma_1}{p_1}\leq s_0 - \frac{d+\gamma_0}{p_0}$.

We next show ${\gamma_1}/{p_1}\leq {\gamma_0}/{p_0}$. Let $f\in \Schw(\R^d)$ be such that $\widehat{f}$ has support in $\{x\in \R^d: |x|<1\}$.
Let $\lambda\geq 1$ and $e_1 = (1, 0,\ldots, 0)$. Then $\supp (\F ( f(\cdot-\lambda e_1)))=  \supp(\widehat{f})$.
By \eqref{nec-embedd} one has
\[\|f(\cdot-\lambda e_1)\|_{B^{s_1}_{p_1, \infty}(\R^d,w_1)} \leq C \|f(\cdot-\lambda e_1)\|_{B^{s_0}_{p_0, 1}(\R^d,w_0)},\]
and as before this yields
\begin{align}\label{eq:estfsupp}
\|f(\cdot-\lambda e_1)\|_{L^{p_1}(\R^d,w_1)}\leq C \|f(\cdot-\lambda e_1)\|_{L^{p_0}(\R^d,w_0)}.
\end{align}
Let $p\in (1, \infty]$, $\gamma> -d$ and $w(x) = |x|^{\gamma}$. We claim that there are constants $c,C>0$, depending on $p,\gamma,d,f$, such that for all $\lambda\geq 1$ one has
\begin{equation}\label{4}
 c \lambda^{\gamma/p} \leq \|f(\cdot-\lambda e_1)\|_{L^{p}(\R^d,w)}\leq  C(1+\lambda^{\gamma})^{1/p}.
\end{equation}
From \eqref{eq:estfsupp} and \eqref{4} it would follow that for all $\lambda\geq 1$ we have
$c \lambda^{\gamma_1/p_1} \leq C (1+\lambda^{\gamma_0})^{1/p_0}$.
Letting $\lambda$ tend to infinity then gives ${\gamma_1}/{p_1}\leq {\gamma_0}/{p_0}$.

Estimate \eqref{4} is trivial for $p=\infty$. To prove \eqref{4} for $p<\infty$, we first consider the case $\gamma\geq 0$. We estimate
\begin{align*}
  \|f(\cdot-\lambda e_1)\|_{L^{p}(\R^d,w)}^{p} & = \int_{\R^d} |f(x-\lambda e_{1})|^{p} |x|^{\gamma} \, dx
  = \int_{\R^d} |f(x)|^{p} |x+\lambda e_1|^{\gamma} \, dx
  \\ & \leq  C_{\gamma} \int_{\R^d} |f(x)|^{p} \big(|x|^{\gamma}  + \lambda^{\gamma} \big) \, dx = C_{p,\gamma,d,f} (1+\lambda^{\gamma}).
\end{align*}
On the other hand,
\begin{align*}
\|f(\cdot-\lambda e_1)\|_{L^{p}(\R^d,w)}^{p} & =  \int_{\R^d} |f(x)|^{p} |x+\lambda e_1|^{\gamma} \, dx
\geq \int_{[0,1]^d} |f(x)|^{p} |x+\lambda e_1|^{\gamma} \, dx
\\ & \geq \int_{[0,1]^d} |f(x)|^{p} \lambda^{\gamma} \, dx = C_{d,p,f} \lambda^{\gamma}.
\end{align*}
Next we consider the case $\gamma<0$. Since $f\in \Schw(\R^d)$, there is a $C$ such that one has
$|f(x)|^p\leq C(1+|x|)^{-d p}$ for $x\in \R^d$. By Lemma \ref{lem:Tx} we can estimate
\begin{align*}
\|f(\cdot-\lambda e_1)\|_{L^{p}(\R^d,w)}^{p} & \leq C \int_{\R^d}  (1+|x|)^{-dp} |x+\lambda e_1|^{\gamma} \, dx
\\ &  = C \int_{\R^d}  (1+|y-\lambda e_1|)^{-dp} |y|^{\gamma} \, dx  \leq C\int_{B(\lambda e_1, 1)} |y|^{\gamma}\, dy
\leq C(1+\lambda)^{\gamma}
\end{align*}
for $\lambda\geq 1$. For the lower estimate we have
\begin{align*}
\|f(\cdot-\lambda e_1)\|_{L^{p}(\R^d,w)}^{p} & \geq \int_{[0,1]^d} |f(x)|^{p} |x+\lambda e_1|^{\gamma} \, dx
\\ & \geq \int_{[0,1]^d} |f(x)|^{p} (1+\lambda)^{\gamma} \, dx = C_{d,p,f} (1+\lambda)^{\gamma} \geq C_{d,p,f,\gamma} \lambda^{\gamma}.
\end{align*}
This completes the proof of \eqref{4} and therefore the proof of ${\gamma_1}/{p_1}\leq {\gamma_0}/{p_0}$.

Next let $\varphi$ be as in Definition \ref{def:Phi} and for each $t>0$ define
$f_t:\R^d\to \C$ by $f_{t}(x) := t^{n} \varphi(t x)$. By \eqref{nec-embedd} one has
$\|f_t\|_{B^{s_1}_{p_1, \infty}(\R^d,w_1)} \leq C \|f_t\|_{B^{s_0}_{p_0, 1}(\R^d,w_0)}$.
Taking $t>0$ small enough it follows that
\[\|f_t\|_{L^{p_1}(\R^d,w_1)}\leq C \|f_t\|_{L^{p_0}(\R^d,w_0)}.\]
Rescaling gives
\[t^{-\frac{\gamma_1 +d}{p_1}}\|\varphi\|_{L^{p_1}(\R^d,w_1)}\leq C t^{- \frac{\gamma_0 + d}{p_0}} \|\varphi\|_{L^{p_0}(\R^d,w_0)}.\]
Letting $t\downarrow 0$ implies that $\frac{\gamma_1 +d}{p_1}\leq
\frac{\gamma_0 + d}{p_0}$.\end{proof}

\begin{remark} In the above proof, the assumption $p_0,p_1 > 1$ was only employed to show $\frac{\gamma_1}{p_1}\leq \frac{\gamma_0}{p_0}$ in case $\gamma_0 <0$. \end{remark}

For  $p_1<p_0$ we can sharpen the necessary condition \eqref{eq:necparameters} for an embedding.

\begin{proposition}\label{prop:neccnonopt2}
Let $1<p_1<p_0 <\infty$ and $s_0,s_1\in \R$. Let
further $w_0(x) = |x|^{\gamma_0}$ and $w_1(x) = |x|^{\gamma_1}$ with
$\gamma_0,\gamma_1>-d$. Suppose
\begin{equation}\label{nec-embedd2}
 B^{s_0}_{p_0,1}(\R^d,w_0) \hookrightarrow B^{s_1}_{p_1, \infty}(\R^d,w_1).
\end{equation}
Then
\begin{equation}\label{eq:necparameters2}
s_0 - \frac{d+\gamma_0}{p_0} \geq s_1 - \frac{d+\gamma_1}{p_1} \qquad
\text{and} \qquad  \frac{\gamma_1+d}{p_1}<\frac{\gamma_0+d}{p_0}.
\end{equation}
\end{proposition}
\begin{remark}Observe here that $\frac{\gamma_1}{p_1} < \frac{\gamma_0}{p_0}$ is already a consequence of $\frac{\gamma_1+d}{p_1}<\frac{\gamma_0+d}{p_0}$ and $p_1<p_0$. \end{remark}

To prove the proposition, we need the following density result. Observe that the proof heavily depends on the fact that the weight is of power type.
\begin{lemma}\label{lem:densityFouriercompact}
Let $p\in [1, \infty)$, $\gamma>-d$ and let $w(x) = |x|^{\gamma}$. The set
\[\F C^\infty_c(\R^d) := \{f\in \Schw(\R^d):\; \widehat{f} \, \emph{\text{ has compact support}}\}\]
is dense in $L^p(\R^d,w)$.
\end{lemma}
\begin{proof}
By Lemma \ref{density}, for the assertion it suffices to consider $f\in \Schw(\R^d)$. We construct a sequence $(f_n)_{n\geq 0}$ of functions in $\F C^\infty_c(\R^d)$ such that $f = \lim_{n\to \infty} f_n$ in $L^p(\R^d,w)$. We proceed as in \cite[Section 6]{MWY}, where a much stronger result has been obtained for the one-dimensional setting.
Let $\zeta\in \Schw(\R^d)$ be such that $\widehat{\zeta}(\xi) = 1$  if $|\xi|\leq 1$, and $\widehat{\zeta}(\xi) = 0$ if $|x|\geq 2$. Let $\zeta_n = n^d \zeta(n x)$ and $f_n = \zeta_n*f$. Choose an integer $k\geq 0$ so large that $ - 2kp + \gamma<-d$. Observe that
\begin{align*}
\|f- f_n\|_{L^p(\R^d,w)} & \leq \sup_{x\in \R^d}|(|x|^{2k}+1) (f(x)-f_n(x) )| \Big(\int_{\R^d} \frac{|x|^{\gamma}}{(|x|^{2k}+1)^p} \, dx\Big)^{1/p}
\\ & \leq C \sup_{x\in \R^d}|(|x|^{2k}+1) (f(x)-f_n(x) )|
\\ & \leq C\| [(-\Delta)^k +1] (\wh{f}-\wh{f}_n)\|_{L^1(\R^d)}
\\ & \leq C\| \Delta^k (\wh{f}-\wh{f_n})\|_{L^1(\R^d)} + \|(\wh{f}-\wh{f}_n)\|_{L^1(\R^d)},
\end{align*}
where we used $\F\big[(|\cdot|^{2k}+1) (f-f_n)\big] = ((-\Delta)^k +1) (\hat{f} - \hat{f}_n)$.
It suffices to show that for any multiindex $a$ one has
\[\lim_{n\to \infty}\|D^{a}(\wh{f}-\wh{f_n})\|_{L^1(\R^d)} =0.\]
Note that $D^{a}[\widehat{f}-\widehat{f}_n] = D^{a}[\widehat{f} (1-\widehat{\zeta}_n)]$. From the Leibniz rule we see that $D^{a}(\widehat{f}-\widehat{f_n})$ consists of finitely many terms of the form $D^{b}\widehat{f} \, D^c (1-\widehat{\zeta}_n)$, where $b,c$ are multiindices. One has
\[\| D^c (1-\widehat{\zeta}_n)\|_{L^1(\R^d)} \leq  \|D^{b}\widehat{f}\|_{L^\infty(\R^d)} \| D^c (1-\widehat{\zeta}_n)\|_{L^1(\R^d)},\]
and the latter converges to zero as $n$ tends to infinity by the dominated convergence theorem.
\end{proof}

\begin{proof}[Proof of Proposition \ref{prop:neccnonopt2}]
In Proposition \ref{prop:neccnonopt} we have already seen that
\[s_0 - \frac{d+\gamma_0}{p_0} \geq s_1 - \frac{d+\gamma_1}{p_1}, \qquad
\text{and} \qquad  \frac{\gamma_1+d}{p_1}\leq \frac{\gamma_0+d}{p_0}
\]
Assume that $\frac{\gamma_1+d}{p_1}=\frac{\gamma_0+d}{p_0}$ holds true. We show that this leads to a contradiction. Let $f\in L^{p_0}(\R^d;w_0)$ be such that
$\supp(\widehat{f})\subset \{\xi\in \R^d:|\xi|\leq 1\}$. Then it follows from
\eqref{nec-embedd2} that \eqref{eq:twoweightnikolskii} holds with $t=1$. By
scaling we see that for all $f\in L^{p_0}(\R^d,w_0)$ such that $\widehat{f}$
has compact support one has
\begin{equation}\label{eq:estwithdelta0}
\|f\|_{L^{p_1}(\R^d,w_1)}\leq C \|f\|_{L^{p_0}(\R^d,w_0)}.
\end{equation}
From Lemma \ref{lem:densityFouriercompact} we see that \eqref{eq:estwithdelta0} extends to all $f\in L^{p_0}(\R^d,w_0)$.
Now define $f:\R^d\to \R$ by $f(x) = |x|^{-d/p_0}
\log(1/|x|)^{-1/p_1} \one_{[0,1/2]}(|x|)$. Then using polar coordinates one easily checks that $f\in
L^{p_0}(\R^d,w_0)$, but $f\notin L^{p_1}(\R^d,w_1)$ which contradicts
\eqref{eq:estwithdelta0}.
\end{proof}

We can finish the proof of the necessary conditions for the embeddings of Besov spaces.
\begin{proof}[Proof of Theorem \ref{thm:main1} (1) $\Rightarrow$ (2)]
It suffices to consider $X=\C$. It follows from \eqref{eq:introBesov} and Proposition \ref{prop:elementaryembeddingBTL} that
$$B_{p_0,1}^{s_0}(\R^d,w_0) \hookrightarrow B_{p_1,\infty}^{s_1}(\R^d,w_1).$$
From Proposition \ref{prop:necc} we see that \eqref{eq:necparameters} holds.
Now there are two possibilities: either (i) $\frac{d+\gamma_1}{p_1} < \frac{d+\gamma_0}{p_0}$, or (ii) $\frac{d+\gamma_1}{p_1} = \frac{d+\gamma_0}{p_0}$.

Suppose that (i) holds. If $s_0 - \frac{d+\gamma_0}{p_0} > s_1 - \frac{d+\gamma_1}{p_1}$, then \eqref{cond:exp} follows. If $s_0 - \frac{d+\gamma_0}{p_0} = s_1 - \frac{d+\gamma_1}{p_1}$, then to obtain \eqref{cond:expnew} we have to show that $q_0 \leq q_1$. Let $(\varphi_n)_{n\geq 0}$ be as in Definition \ref{def:Phi}. For a sequence $(a_j)$ and $N\in \N$, define the function $f = \sum_{j=1}^N 2^{-3j(d + s_0 - \frac{d+\gamma_0}{p_0})}a_j \varphi_{3j}$. We have $\varphi_n * \varphi_{3j}\neq 0$ only for $n= 3j + l $ with $l\in \{-1,0,1\}$, and by Lemma \ref{lem:phinphij}, $\|\varphi_{3j+l} * \varphi_{3j}\|_{L^{p}(\R^d,w)} = C\, 2^{3j d} 2^{-3j \frac{d+\gamma}{p}}$. It follows that
$$ \|(a_j)_{j\leq N}\|_{\ell^{q_1}} \leq C  \|(a_j)_{j\leq N}\|_{\ell^{q_0}},$$
with a constant $C$ independent of $N$ and $a_j$. But this is only possible for $q_0 \leq q_1$.

Now suppose that (ii) holds. Then $\frac{\gamma_1}{p_1}\leq \frac{\gamma_0}{p_0}$ yields $p_0\geq p_1$. If $p_0>p_1$, then Proposition \ref{prop:neccnonopt2} yields $\frac{d+\gamma_1}{p_1} < \frac{d+\gamma_0}{p_0}$ and this contradicts (ii). If $p_0 = p_1$, then $\gamma_0 = \gamma_1$ follows from (ii) and therefore, $s_0\geq s_1$ by \eqref{eq:necparameters}. If $s_0 = s_1$, then it follows as above that $q_0 \leq q_1$. Hence \eqref{cond:trivial} is valid.\end{proof}

\section{Embeddings for Triebel-Lizorkin spaces -- Proof of Theorem \ref{thm:main2}\label{sec:TriebelLizorkin}}

\subsection{Sufficient conditions} In the proof of the sufficiency part of Theorem \ref{thm:main2} we employ ideas from \cite{BM01} and \cite{SchmSi05}. One has the following Gagliardo-Nirenberg type inequality for spaces with weights.

\begin{proposition}\label{prop:interpolationineq} Let $X$ be a Banach space, $q, q_0, q_1\in [1,\infty]$ and $\theta\in (0,1)$. Let $p,p_0, p_1\in (1, \infty)$ and  $-\infty<s_0<s_1<\infty$  satisfy
\[\frac{1}{p} =  \frac{1-\theta}{p_0} + \frac{\theta}{p_1} \ \ \ \text{and} \ \ \ s = (1-\theta)s_0 + \theta s_1.\]
Let further $w, w_0, w_1\in A_{\infty}$ be such that $w= w_0^{(1-\theta)p/p_0} w_1^{\theta p/p_1}$.
Then there exists a constant $C$ such that for all $f\in \TD(\R^d;X)$ one has
\[\|f\|_{F^{s}_{p,q}(\R^d,w;X)} \leq C \|f\|_{F^{s_0}_{p_0,q_0}(\R^d, w_0;X)}^{1-\theta} \|f\|_{F^{s_1}_{p_1,q_1}(\R^d, w_1;X)}^\theta.\]
In particular, one can take $w = w_0 = w_1$.
\end{proposition}

\begin{proof} Due to \cite[Lemma 3.7]{BM01}, for any sequence of scalars $(a_j)_{j\geq 0}$ one has
\[\|(2^{sj} a_j)_{j\geq 0}\|_{\ell^q}\leq \|(2^{s_0j} a_j)_{j\geq 0}\|_{\ell^\infty}^{1-\theta} \|(2^{s_1j} a_j)_{j\geq 0}\|_{\ell^\infty}^{\theta}.\]
 Taking $a_j(x) = \|S_j f(x)\|$ with $x\in \R^d$, one obtains
\[\|(2^{sj} a_j(x))_{j\geq 0}\|_{\ell^q}^p w(x)\leq  \|(2^{s_0j} a_j(x))_{j\geq 0}\|_{\ell^\infty}^{(1-\theta)p} w_0(x)^{(1-\theta)p/p_0} \, \|(2^{s_1j} a_j(x))_{j\geq 0}\|_{\ell^\infty}^{\theta p} w_1(x)^{\theta p/p_1}.\]
Thus H\"older's inequality gives
\[\|f\|_{F^{s}_{p,q}(\R^d,w;X)} \leq C \|f\|_{F^{s_0}_{p_0,\infty}(\R^d,w_0;X)}^{1-\theta} \|f\|_{F^{s_1}_{p_1,\infty}(\R^d,w_1;X)}^\theta,\]
and  the assertion follows from Proposition \ref{prop:elementaryembeddingBTL}.
\end{proof}

We turn to the proof of sufficiency.

\begin{proof}[Proof of Theorem \ref{thm:main2} (2) $\Rightarrow$ (1)]
By the elementary embeddings of Proposition \ref{prop:elementaryembeddingBTL},
one can assume that $s_0 -\frac{d+\gamma_0}{p_0}= s_1 -
\frac{d+\gamma_1}{p_1}$ and $q_1 = 1$. The trivial cases in (2) are also covered Proposition \ref{prop:elementaryembeddingBTL}. We thus have to show that \eqref{cond:exp} implies embedding (1).

Let $\theta_0\in [0,1)$ be such that
$\frac{1}{p_1} - \frac{1-\theta_0}{p_0} =0$. Consider the function
$g:(\theta_0, 1]\to \R$ given by
\[g(\theta) = \frac{\gamma_1/p_1 -(1-\theta)\gamma_0/p_0}{\frac{1}{p_1} - \frac{1-\theta}{p_0}}.\]
Obviously, $g$ is continuous, and $\lim_{\theta\uparrow 1} g(\theta) = \gamma_1$. Since $\gamma_1>-d$ we can choose a $\theta\in (\theta_0,1)$ such that $\gamma:=g(\theta)>-d$. Let $v(x) = |x|^\gamma$, and let $r$ be defined by
$\frac{1}{p_1} = \frac{1-\theta}{p_0} + \frac{\theta}{r}.$
Note that $p_0\leq p_1 < \infty$ implies $r\in [p_1,\infty)$. Let further $t$ be defined by
$t - \frac{d+\gamma}{r} = s_1  - \frac{d+\gamma_1}{p_1}.$
Observe that $t < s_0$, $s_1 = \theta t + (1-\theta) s_0$ and $v^{p_1\theta/r} w_0^{(1-\theta)p_1/p_0} = w_1.$
Therefore, by Proposition \ref{prop:interpolationineq},
\begin{equation}\label{eq:intFhelp}
\|f\|_{F^{s_1}_{p_1,1}(\R^d,w_1;X)}\leq C \|f\|_{F^{s_0}_{p_0,q_0}(\R^d,w_0;X)}^{1-\theta} \|f\|_{F^{t}_{r,r}(\R^d,v;X)}^\theta.
\end{equation}
Now one can check that
$$\frac{\gamma_1}{p_1} - \frac{\gamma}{r} = \frac{1-\theta}{\theta} \Big( \frac{\gamma_0}{p_0} - \frac{\gamma_1}{p_1} \Big)\geq 0.$$
From Proposition \ref{prop:elementaryembeddingBTL} and Theorem \ref{thm:main1} (using $r\geq p_1$) one obtains that
\begin{align*}
 \|f\|_{F^{t}_{r,r}(\R^d,v;X)} &\, = \|f\|_{B^{t}_{r,r}(\R^d,v;X)} \leq C \|f\|_{B^{s_1}_{p_1,p_1}(\R^d,w_1;X)} \leq C \|f\|_{F^{s_1}_{p_1,1}(\R^d,w_1;X)}.
\end{align*}
Substituting the latter estimate in \eqref{eq:intFhelp}, one deduces that
\[\|f\|_{F^{s_1}_{p_1,1}(\R^d,w_1;X)} \leq C\|f\|_{F^{s_0}_{p_0,q_0}(\R^d,w_0;X)}.\]
\end{proof}

\subsection{Necessary conditions} The necessary conditions for the $F$-spaces are a direct consequence of the result for the $B$-spaces.

\begin{proof}[Proof of Theorem \ref{thm:main2} (1) $\Rightarrow$ (2)]
Assume (1). It suffices to consider $X=\C$. By Proposition \ref{prop:elementaryembeddingBTL} one has
\[B^{s_0}_{p_0,1}(\R^d,w_0)\hookrightarrow F^{s_0}_{p_0,q_0}(\R^d,w_0) \hookrightarrow F^{s_1}_{p_1,q_1}(\R^d,w_1)\hookrightarrow B^{s_1}_{p_1,\infty}(\R^d,w_1).\]
Therefore, (2) follows from the Propositions \ref{prop:necc} and \ref{prop:elementaryembeddingBTL}.
\end{proof}

Now we can prove the characterization for the $H$- and $W$-spaces in case $p_0\leq p_1$.

\begin{proof}[Proof of Corollaries \ref{cor:introHspaces} and \ref{cor:introWspaces}]
This follows from Theorem \ref{thm:main2} and Proposition \ref{prop:squeezeHF}.
\end{proof}

\begin{remark}
It is unclear to us whether Corollaries \ref{cor:introHspaces} and \ref{cor:introWspaces} hold for all $\gamma_0, \gamma_1>-d$. This is contained in \cite{Rab} for $s_0 = 0$, $s_1=0$ in the case of $W$-spaces.
\end{remark}

In Proposition \ref{prop:p1kleinerp0} we give a characterization for the embedding of $H$ and $F$-spaces in case $p_1<p_0$. Its proof will be based on the following result.

\begin{proposition}\label{prop:neccnonopt}
Let $1<p_1<p_0 <\infty$ and $s_0,s_1\in \R$. Let further $w_0(x) = |x|^{\gamma_0}$ and $w_1(x) = |x|^{\gamma_1}$ with $-d< \gamma_0 < d(p_0-1)$ and $-d< \gamma_1 < d(p_1-1)$. Suppose
\begin{equation}\label{nec-embedd3}
H^{s_0,p_0}(\R^d,w_0) \hookrightarrow H^{s_1,p_1}(\R^d,w_1).
\end{equation}
Then
\begin{equation}\label{eq:p1kleinerp0}
\frac{\gamma_1+d}{p_1}< \frac{\gamma_0+d}{p_0} \qquad \text{and} \qquad s_0 - \frac{d+\gamma_0}{p_0} > s_1 - \frac{d+\gamma_1}{p_1}.
\end{equation}
\end{proposition}
\begin{proof}
It suffices to consider the case $X=\C$ and $s_1 = 0$. The Propositions \ref{prop:squeezeHF}, \ref{prop:necc} and \ref{prop:neccnonopt2}  imply that $s_0 - \frac{d+\gamma_0}{p_0} \geq - \frac{d+\gamma_1}{p_1}$ and $\frac{\gamma_1+d}{p_1} < \frac{\gamma_0+d}{p_0}.$
In particular, we have $s_0> 0$.

We suppose that
\[s_0 - \frac{d+\gamma_0}{p_0} =  -\frac{d+\gamma_1}{p_1}, \qquad \frac{\gamma_1+d}{p_1}< \frac{\gamma_0+d}{p_0},\]
and show that this gives a contradiction. It follows from \cite[Proposition VI.4.4/2, Corollary V.4.2]{Stein93} that the operator $(1-\Delta)^{s_0/2} (1+ (-\Delta)^{s_0/2})^{-1}$ is bounded on $L^{p_0}(\R^d,w_0)$.
Thus for $f\in \Schw(\R^d)$, the embedding \eqref{nec-embedd3} implies that
\begin{align*}
 \|f\|_{L^{p_1}(\R^d,w_1)} &\, \leq C \|(1-\Delta)^{s_0/2} f\|_{L^{p_0}(\R^d,w_0)}   \leq C\|f\|_{L^{p_0}(\R^d,w_0)} + \| (-\Delta)^{s_0/2} f\|_{L^{p_0}(\R^d,w_0)}.
\end{align*}
By applying the above estimate to $f(\lambda\cdot)$ for $\lambda>0$, scaling shows that
\[\|f\|_{L^{p_1}(\R^d,w_1)} \leq C \lambda^{-s_0} \|f\|_{L^{p_0}(\R^d,w_0)} + C \|(-\Delta)^{s_0/2} f\|_{L^{p_0}(\R^d,w_0)}.\]
Letting $\lambda\to \infty$ gives
\begin{equation}\label{eq:Rieszineq}
\|f\|_{L^{p_1}(\R^d,w_1)} \leq \tilde{C} \|(-\Delta)^{s_0/2} f\|_{L^{p_0}(\R^d,w_0)}.
\end{equation}
By density it follows that for all $f\in \TD(\R^d)$ for which $(-\Delta)^{s_0/2} f\in L^{p_0}(\R^d,w_0)$, the estimate \eqref{eq:Rieszineq} holds.

Now define the radial function $g:\R^d\to \R$ as $g(x) = |x|^{-a}
\log(1/|x|)^{-b}\one_{[0,1/2]}(|x|)$, where
\[a=s_0 + \frac{\gamma_1+d}{p_1} = \frac{\gamma_0+d}{p_0}, \qquad  b = 1/p_1.\]
One has that $g\in L^{p_0}(\R^d,w_0)$. Indeed, using polar coordinates one
sees that
\[\|g\|_{L^{p_0}(\R^d,w_0)}^{p_0} = c \int_{0}^{1/2} r^{d-1} r^{-a p_0} \log(1/r)^{-b
p_0}\, r^{\gamma_0} \, dr = c \int_{0}^{1/2} r^{-1} \log(1/r)^{- p_0/p_1}\, \,
dr<\infty,
\]
because $p_0/p_1>1$. Let $f\in \TD(\R^d)$ be defined by $f =
(-\Delta)^{-s_0/2} g$. To show that \eqref{eq:Rieszineq} cannot hold, and thus to prove that \eqref{eq:p1kleinerp0}, it
suffices to show that $f \notin L^{p_1}(\R^d,w_1)$. This will be checked for $d\geq 2$ and $d=1$ separately. First assume $d\geq 2$. One has the following representation of the Riesz potential for radial symmetric functions $v:\R^d\to \R$ (see \cite{NDD}):
\[(-\Delta)^{-s_0/2} v(x) = c \int_0^\infty v(r) r^{s_0-1} I_{j, k} (\rho/r) \, dr.\]
Here $j = d-s_0$,  $k = (d-3)/2$ and $\rho = |x|$ and
\[I_{j, k}(z) = \int_{-1}^1 \frac{(1-t^2)^k}{(1-2zt + z^2)^{j/2}} \, dt, \qquad  z\geq 0.\]
The function $I_{j,k}:[0,\infty)\to [0,\infty)$ is continuous on $[0,\infty)\setminus\{1\}$ and its singularity at $1$ is well-understood (see \cite[Lemma 4.2]{NDD}). For $\rho = |x|\leq \frac12$, $\rho \neq 0$, we obtain
\begin{align*}
f(x) = (-\Delta)^{-s_0/2} g(x) & = c \int_0^{1/2} r^{s_0-a-1} \log(1/r)^{-b} I_{j, k} (\rho/r) \, dr
\\ &  \geq c \log(2/\rho)^{-b}  \int_{\rho/2}^{\rho} r^{s_0-a-1} I_{j, k} (\rho/r) \, dr
\\ & = c \rho^{s_0-a}   \log(2/\rho)^{-b}  \int_{1/2}^{1} u^{s_0-a-1} I_{j, k} (1/u) \, dr
\geq C \rho^{s_0-a} \log(2/\rho)^{-b}.
\end{align*}
It follows that
\begin{align*}
\|f\|_{L^{p_1}(\R^d,w_1)}^{p_1} & \geq \int_{|x|\leq \frac12} |f(x)|^{p_1} |x|^{\gamma_1} \, dx
\\ & \geq C \int_{0}^{1/2} \rho^{d-1} \rho^{(s_0-a)p_1} \log(2/\rho)^{-b p_1} \rho^{\gamma_1} \, d\rho
= C \int_{0}^{1/2} \rho^{-1} \log(2/\rho)^{-1} \, d\rho = \infty.
\end{align*}
Hence $f \notin L^{p_1}(\R^d,w_1)$. If $d=1$, then for $x\in [0,1/2]$ one has
\begin{align*}
f(x) = (-\Delta)^{-s_0/2} g(x) & = c \int_{-1/2}^{1/2} |x-y|^{s_0-1} |y|^{-a} \log(1/|y|)^{-b} \, dy
\\ & \geq  \log(2/x)^{-b}  \int_{x/2}^{x} |x-y|^{s_0-1} |y|^{-a} \, dy
 = c x^{s_0-a} \log(2/x)^{-b}.
\end{align*}
Now the proof can be finished as before.
\end{proof}

We obtain the following consequences for the $F$- and the $W$-spaces.

\begin{corollary}\label{cor:Fcasecounter}
Let $1<p_1<p_0 <\infty$ and $s_0,s_1\in \R$. Let further $w_0(x) = |x|^{\gamma_0}$ and $w_1(x) = |x|^{\gamma_1}$ with $-d< \gamma_0 < d(p_0-1)$ and $-d < \gamma_1< d(p_1-1)$. Suppose that for some $q_0\in [2, \infty]$ and $q_1\in [1, 2]$ one has
\begin{equation}\label{nec-embedd3F}
F^{s_0}_{p_0,q_0}(\R^d,w_0) \hookrightarrow F^{s_1}_{p_1,q_1}(\R^d,w_1).
\end{equation}
Then \eqref{eq:p1kleinerp0} holds.
\end{corollary}
\begin{proof}
By Proposition \ref{prop:lifting} and the results in \cite[Section 4]{Bui82} we have
\begin{equation}\label{eq:FH}
F^{s_i}_{p_i,2}(w_i) = H^{s_i,p_i}(w_i) \qquad \text{for $i=1, 2$.}
\end{equation} Therefore, by Proposition \ref{prop:elementaryembeddingBTL}, \eqref{nec-embedd3F} implies \eqref{nec-embedd3}. Now the result follows from Proposition \ref{prop:neccnonopt}.
\end{proof}

\begin{corollary}\label{cor:Wcasecounter}
Let $1<p_1<p_0 <\infty$ and $s_0,s_1\in \N_0$. Let further $w_0(x) = |x|^{\gamma_0}$ and $w_1(x) = |x|^{\gamma_1}$ with $-d< \gamma_0 < d(p_0-1)$ and $-d < \gamma_1< d(p_1-1)$. Suppose
\begin{equation}\label{nec-embedd3W}
W^{s_0,p_0}(\R^d,w_0) \hookrightarrow W^{s_1,p_1}(\R^d,w_1).
\end{equation}
Then \eqref{eq:p1kleinerp0} holds.
\end{corollary}
\begin{proof} Since $L^{p_i}(\R^d,w_i) = F^{0}_{p_i,2}(\R^d,w_i)$ by \cite[Theorem 1.10]{Ry01},  Proposition \ref{prop:differentiation} implies that $W^{s_i,p_i}(\R^d,w_i) = F^{s_i}_{p_i,2}(\R^d,w_i)$ for $i=1,2$.  Now the result follows from Corollary \ref{cor:Fcasecounter}. \end{proof}

We end this section with the characterization of embeddings for $H$-spaces in case $p_0 > p_1$.

\begin{proof}[Proof of Proposition \ref{prop:p1kleinerp0}]
To prove (1) $\Rightarrow$ (3) and (2) $\Rightarrow$ (3), it suffices to consider $X=\C$. Note that because of \eqref{eq:FH} it suffices to prove (1) $\Rightarrow$ (3), since (1) and (2) coincide for $X=\C$. Now (3) follows from Proposition \ref{prop:neccnonopt}.

We prove the sufficiency part. Assume (3). By Theorem \ref{thm:main1} it follows that $B^{s_0}_{p_0, \infty}(\R^d,w_0; X)\hookrightarrow B^{t_1}_{p_1, \infty}(\R^d,w_1; X)$, where $t_1$ is defined by $s_0 - \frac{\gamma_0+d}{p_0}  = t_1 - \frac{\gamma_1+d}{p_1}$. Note that it follows from (3) that $t_1>s_1$. Therefore, combining the above embedding with Proposition \ref{prop:elementaryembeddingBTL} yields
\[B^{s_0}_{p_0, \infty}(\R^d,w_0; X)\hookrightarrow B^{s_1}_{p_1, 1}(\R^d,w_1; X).\]
Now (1) and (2) are a consequence of the Propositions \ref{prop:elementaryembeddingBTL} and \ref{prop:squeezeHF}.
\end{proof}

\section{Embeddings of Jawerth--Franke type}
In this section we only treat power weights of $A_p$-type. We first consider  real interpolation of the weighted function spaces.

\begin{proposition}\label{thm:realinter}
Let $X$ be a Banach space, $p\in (1, \infty)$, $q_0, q_1,q\in[1, \infty]$, $s_0\neq s_1\in \R$. Then for $\theta\in [0,1]$ and  $s = (1-\theta) s_0 + \theta s_1$ and $w\in A_p$ one has
\begin{align}
\label{eq:Besovinterpolation} (B^{s_0}_{p,q_0}(\R^d,w;X), B^{s_1}_{p,q_1}(\R^d,w;X))_{\theta,q} &= B^{s}_{p,q}(\R^d,w;X),
\\\label{eq:TLinterpolation}  (F^{s_0}_{p,q_0}(\R^d,w;X), F^{s_1}_{p,q_1}(\R^d,w;X))_{\theta,q} &= B^{s}_{p,q}(\R^d,w;X),
\\ \label{eq:Besselinterpolation}(H^{s_0,p}(\R^d,w;X), H^{s_1,p}(\R^d,w;X))_{\theta,q} &= B^{s}_{p,q}(\R^d,w;X) ,
\end{align}
and if additionally $s_0,s_1\geq 0$ are integers, then
\begin{align}
\label{eq:Sobolevinterpolation} (W^{s_0,p}(\R^d,w;X), W^{s_1,p}(\R^d,w;X))_{\theta,q} = B^{s}_{p,q}(\R^d,w;X).
\end{align}
Moreover, for $p_0, p_1\in (1, \infty)$, $s\in \R$, $\theta\in (0,1)$, $\frac1p = \frac{1-\theta}{p_0} + \frac{\theta}{p_1}$, $q\in (1,\infty]$ and $w_0 \in A_{p_0}$, $w_1\in A_{p_1}$, $w\in A_p$ with $w^{1/p} = w_0^{(1-\theta)/p_0} w_1^{\theta/p_1}$ it holds that
\begin{align}\label{eq:TrLinterp}
(F^{s}_{p_0,q}(\R^d,w_0;X), F^{s}_{p_1,q}(\R^d,w_1;X))_{\theta,p} &= F^{s}_{p,q}(\R^d,w;X),\\
(F^{s}_{p_0,1}(\R^d,w_0;X), F^{s}_{p_1,1}(\R^d,w_1;X))_{\theta,p} &\hookrightarrow F^{s}_{p,1}(\R^d,w;X).\label{eq:TrLinterp1}
\end{align}
\end{proposition}

\begin{proof} Since all the weights under consideration are of class $A_p$, the proofs are straightforward generalizations of the unweighted case, as presented in \cite[Section 2.4]{Tr1} and \cite[Proposition 12]{SchmSiunpublished}, for instance.

We nevertheless provide the details for (\ref{eq:TrLinterp}) and (\ref{eq:TrLinterp1}). By Proposition \ref{prop:lifting} it suffices to consider $s=0$, respectively.
Let $(\varphi_n)_{n\geq 0} \in \Phi$. For $p_*\in (1,\infty)$, $q_*\in [1,\infty]$ and $w_*\in A_{p_*}$ it follows from the definitions that the map
$$S: F_{p_*,q_*}^0(\R^d,w_*;X) \to L^{p_*}(\R^d, w_*;\ell^{q_*}(X)), \qquad S f = (\varphi_n * f)_{n\geq 0},$$
is continuous. If $(\psi_n)_{n\geq 0}$ is as in the proof of Proposition \ref{prop:lifting}, then Lemma \ref{lem:Stein} implies that
$$R: L^{p_*}(\R^d, w_*;\ell^{q_*}(X)) \to F_{p_*,q_*}^0(\R^d,w_*;X), \qquad R(g_k)_{k\geq 0} = \sum_{k=0}^\infty \psi_k* g_k,$$
is continuous for $q\in (1,\infty]$. Now if $p_0,p_1,p$ and $w_0,w_1,w$ are as in \eqref{eq:TrLinterp} and \eqref{eq:TrLinterp1}, then \cite[Theorem 1.18.5]{Tr1} gives
\begin{equation}\label{1}
 (L^{p_0}(\R^d, w_0;\ell^q(X)), L^{p_1}(\R^d, w_1;\ell^q(X)))_{\theta,p} = L^p(\R^d, w;\ell^q(X)), \qquad q\in [1,\infty].
\end{equation}
Since $R$ is a right-inverse for $S$, (\ref{eq:TrLinterp}) is a consequence of the well-known retraction-corectraction method (see \cite[Theorem 1.2.4]{Tr1}). Moreover, (\ref{1}) implies that
$$S: (F^{0}_{p_0,1}(\R^d,w_0;X), F^{0}_{p_1,1}(\R^d,w_1;X))_{\theta,p} \to L^p(\R^d, w;\ell^1(X))$$
is continuous as well. Thus \eqref{eq:TrLinterp1} follows from
\begin{equation}\label{Finteremb}
 \|f\|_{F_{p,1}^0(\R^d,w;X)} = \|Sf\|_{L^p(\R^d, w;\ell^1(X))} \leq C\, \|f\|_{(F^{0}_{p_0,1}(\R^d,w_0;X), F^{0}_{p_1,1}(\R^d,w_1;X))_{\theta,p}}.
\end{equation}

Employing Lemma \ref{lem:Stein} and interpolation theory for a class of weighted $\ell^q$-spaces (see \cite[Theorem 1.18.2]{Tr1}), one can show (\ref{eq:Besovinterpolation}) in a similiar way. Finally, the identities (\ref{eq:TLinterpolation}), \eqref{eq:Besselinterpolation} and \eqref{eq:Sobolevinterpolation} follow from the independence of (\ref{eq:Besovinterpolation}) of the microscopic parameters $q_0,q_1\in [1,\infty]$ and the Propositions \ref{prop:elementaryembeddingBTL} and \ref{prop:squeezeHF}.\end{proof}

\begin{remark} \label{interAinfty} The operator $S$ in the proof above is continuous for all $w\in A_\infty$. It thus follows from \eqref{Finteremb} that the embeddings from the left to the right in \eqref{eq:TrLinterp} and \eqref{eq:TrLinterp1} are true for $w\in A_\infty$.
\end{remark}

\begin{remark}
 In \cite{Bui82}, interpolation results for scalar $B$- and $F$-spaces are shown for a single weight $w\in A_\infty$. Even more general results and different proofs are given in \cite[Theorem 2.14]{Ry01}.
\end{remark}

After these preparations we can show embeddings of Jawerth--Franke type, which is an improvement of the embeddings in \eqref{eq:introBesov} and \eqref{eq:introF}. We argue similiar to \cite[Theorem 6]{SchmSiunpublished}.

\begin{theorem}\label{thm:sobolevJawerth}
Let $X$ be a Banach space,  $1<p_0<p_1<\infty$, $s_0,s_1\in \R$, $q\in [1, \infty]$ and $w_0(x) = |x|^{\gamma_0}$, $w_1(x) = |x|^{\gamma_1}$ with $\gamma_0\in (-d,d(p_0-1))$ and $\gamma_1\in (-d,d(p_1-1))$. Suppose
\begin{equation}\label{eq:JawerthFranke}
\frac{\gamma_1}{p_1} \leq \frac{\gamma_0}{p_0} \quad \text{and} \quad s_0 - \frac{d+\gamma_0}{p_0} \geq s_1 - \frac{d+\gamma_1}{p_1}.
\end{equation}
Then
\begin{align}
\label{eq:JF1} B^{s_0}_{p_0,p_1}(\R^d,w_0;X) & \hookrightarrow F^{s_1}_{p_1,q}(\R^d,w_1;X),\\
\label{eq:JF2} F^{s_0}_{p_0,q}(\R^d,w_0;X) & \hookrightarrow B^{s_1}_{p_1,p_0}(\R^d,w_1;X).
\end{align}

\end{theorem}

\begin{proof} For (\ref{eq:JF1}), by Proposition \ref{prop:elementaryembeddingBTL} it suffices to consider the case $q=1$ and $s_0 -\frac{d+\gamma}{p_0} = s_1 - \frac{d+\gamma}{p_1}.$ Let $r_0,r_1\in(1, \infty)$ be such that $p_0<r_0<p_1<r_1$, and let $\mu_0 \in (-d, d(r_0-1))$ and  $\mu_1 \in (-d, d(r_1-1))$  satisfy $\frac{\mu_0}{r_0} = \frac{\mu_1}{r_1} = \frac{\gamma_1}{p_1}.$ Let further $\theta\in (0,1)$ and $t_0, t_1\in \R$ be such that
\[\frac{1}{p_1} = \frac{1-\theta}{r_0} + \frac{\theta}{r_1}, \qquad t_0 -\frac{d + \gamma_0}{p_0} = s_1 -\frac{d+\mu_0}{r_0} , \qquad t_1 -\frac{d + \gamma_0}{p_0} = s_1 -\frac{d+\mu_1}{r_1} .\]
Since $(1-\theta) t_0 + \theta t_1 = s_0$, we obtain from (\ref{eq:Besovinterpolation}) that
$$
B^{s_0}_{p_0,p_1}(\R^d,w_0;X)  = (F^{t_0}_{p_0,1}(\R^d,w_0;X),  F^{t_1}_{p_0,1}(\R^d,w_0;X))_{\theta,p_1}.
$$
Setting $v_0(x) = |x|^{\mu_0}$ and $v_1(x) = |x|^{\mu_1}$, Theorem \ref{thm:main2} gives the embeddings
\[F^{t_0}_{p_0,1}(\R^d,w_0;X) \hookrightarrow F^{s_1}_{r_0,1}(\R^d,v_0;X), \qquad F^{t_1}_{p_0,1}(\R^d,w_0;X) \hookrightarrow F^{s_1}_{r_1,1}(\R^d,v_1;X),\]
due to the definition of $t_0, t_1$ and $\frac{\mu_0}{r_0} = \frac{\mu_1}{r_1} = \frac{\gamma_1}{p_1}\leq  \frac{\gamma_0}{p_0}$. Therefore
$$B^{s_0}_{p_0,p_1}(\R^d,w_0;X) \hookrightarrow (F^{s_1}_{r_0,1}(\R^d,v_0;X), F^{s_1}_{r_1,1}(\R^d,v_1;X))_{\theta, p_1}  \hookrightarrow F^{s_1}_{p_1,1}(\R^d,w_1;X),$$
where the latter embedding follows from \eqref{eq:TrLinterp1}, as $\frac{1}{p_1} = \frac{1-\theta}{r_0} + \frac{\theta}{r_1}$ and  $\frac{\gamma_1}{p_1} = (1-\theta)\frac{\mu_0}{r_0} + \theta \frac{\mu_1}{r_1}$.

To show (\ref{eq:JF2}), as above it suffices to consider $q=\infty$ and $s_0 - \frac{d+\gamma}{p_0} = s_1 -\frac{d+\gamma}{p_1}$. Let $r_0,r_1\in(1, \infty)$, $\theta\in (0,1)$, $\mu_0 \in (-d, d(r_0-1))$ and  $\mu_1 \in (-d, d(r_1-1))$ be such that
\[r_0<p_0<r_1<p_1, \qquad \frac{1}{p_0} = \frac{1-\theta}{r_0} + \frac{\theta}{r_1}, \qquad \frac{\mu_0}{r_0} = \frac{\mu_1}{r_1} = \frac{\gamma_0}{p_0}.\]
Setting again  $v_0(x) = |x|^{\mu_0}$ and $v_1(x) = |x|^{\mu_1}$, it follows from \eqref{eq:TrLinterp} that
$$F^{s_0}_{p_0,\infty}(\R^d,w_0;X)  = (F^{s_0}_{r_0,\infty}(\R^d,v_0;X) ,  F^{s_0}_{r_1,\infty}(\R^d,v_1;X) )_{\theta,p_0}.$$
Let  the numbers $t_0, t_1\in \R$ be defined by
$$s_0 -\frac{d+\mu_0}{r_0} = t_0 -\frac{d + \gamma_1}{p_1} , \qquad s_0 -\frac{d+\mu_1}{r_1} = t_1 -\frac{d + \gamma_1}{p_1}.$$
Using that $\frac{\gamma_1}{p_1} \leq \frac{\gamma_0}{p_0} = \frac{\mu_0}{r_0} = \frac{\mu_1}{r_1}$,  Proposition \ref{prop:elementaryembeddingBTL} and Theorem \ref{thm:main1} yield
\[F^{s_0}_{r_0,\infty}(\R^d,v_0;X) \hookrightarrow B^{s_0}_{r_0,\infty}(\R^d,v_0;X) \hookrightarrow B^{t_0}_{p_1,\infty}(\R^d,w_1;X),\]
\[F^{s_0}_{r_1,\infty}(\R^d,v_1;X) \hookrightarrow B^{s_0}_{r_1,\infty}(\R^d,v_1;X) \hookrightarrow B^{t_1}_{p_1,\infty}(\R^d,w_1;X).\]
We thus have
\begin{align*}
F^{s_0}_{p_0,\infty}(\R^d,w_0;X) \hookrightarrow (B^{t_0}_{p_1,\infty}(\R^d,w_1;X), B^{t_1}_{p_1,\infty}(\R^d,w_1;X))_{\theta, p_0} = B^{s_1}_{p_1,p_0}(\R^d,w_1;X),
\end{align*}
where the latter identity follows from $(1-\theta) t_0 + \theta t_1 = s_1$ and (\ref{eq:Besovinterpolation}).\end{proof}

\begin{remark}
For $w_0=w_1\in A_\infty$, the scalar versions of (\ref{eq:JF1}) and (\ref{eq:JF2}) are shown in \cite[Theorem 2.6]{Bui82} and \cite[Proposition 1.8]{HaSkr}.
\end{remark}

\begin{remark}\label{discJF} In the scalar case, the interpolation identities \eqref{eq:Besovinterpolation} and \eqref{eq:TLinterpolation} are shown in \cite{Bui82} for $w\in A_\infty$. Inspecting the proof for (\ref{eq:JF1}), we see that only the interpolation embedding \eqref{eq:TrLinterp1} was  used, which is also true for $w\in A_\infty$ by Remark \ref{interAinfty}. Hence \eqref{eq:JawerthFranke} implies (\ref{eq:JF1}) for all weight exponents $\gamma_0, \gamma_1 >-d$.

Also for (\ref{eq:JF2}) we expect that the restrictions $\gamma_0<d(p_0-1)$ and $\gamma_1<d(p_1-1)$ are not necessary. In case $\gamma_1/p_1<\gamma_0/p_0$ one can give a more direct proof which only makes use of \eqref{eq:TrLinterp} for a single weight. The scalar version of \eqref{eq:TrLinterp} is shown in \cite[Theorem 3.5]{Bui82}. Hence \eqref{eq:JawerthFranke} implies (\ref{eq:JF1}) for all $\gamma_0,\gamma_1 >-d$ under these assumptions.

However, the sharp case $\gamma_1/p_1=\gamma_0/p_0$ for (\ref{eq:JF2}) remains open. One needs \eqref{eq:TrLinterp} also for $A_\infty$-weights, which seems to be an open problem (see \cite[Remark 3.4]{Bui82}).
\end{remark}

\section{Embeddings into $L^p$-spaces and H\"older spaces}

In this section we discuss conditions under which weighted spaces of smooth functions embed into function spaces such as $L^p$-spaces and spaces of H\"older continuous functions. The results are consequences of our main results.

\begin{proposition}\label{prop:intoweightedLpB}
Let $X$ be a Banach space, let $1<p_0,p_1 \leq \infty$, $q_0\in [1, p_0]$.
Let $w_0(x) = |x|^{\gamma_0}$ and $w_1(x) = |x|^{\gamma_1}$ with $\gamma_0,\gamma_1>-d$. Assume
\[s_0-\frac{d+\gamma_0}{p_0} \geq -\frac{d+\gamma_1}{p_1}, \qquad \frac{\gamma_1}{p_1} \leq \frac{\gamma_0}{p_0} \qquad \text{and} \qquad \frac{d+\gamma_1}{p_1}<\frac{d+\gamma_0}{p_0}.\]
If $p_0\leq p_1$ or $q_0=1$, then
\[
B^{s_0}_{p_0,q_0}(\R^d,w_0;X)\hookrightarrow L^{p_1}(\R^d,w_1;X),
\]
\end{proposition}

\begin{proposition}\label{prop:intoweightedLpF}
Let $X$ be a Banach space, let $1<p_0\leq p_1 <\infty$, $q_0\in [1, \infty]$. Let $w_0(x) = |x|^{\gamma_0}$ and $w_1(x) = |x|^{\gamma_1}$ with $\gamma_0,\gamma_1>-d$. If
\[s_0-\frac{d+\gamma_0}{p_0} \geq -\frac{d+\gamma_1}{p_1},  \qquad \text{and} \qquad  \frac{\gamma_1}{p_1} \leq \frac{\gamma_0}{p_0},\]
then
\[
F^{s_0}_{p_0,q_0}(\R^d,w_0;X)\hookrightarrow L^{p_1}(\R^d,w_1;X),
\]
\end{proposition}

\begin{proof}[Proof of Propositions \ref{prop:intoweightedLpB} and \ref{prop:intoweightedLpF}]
Let $p_0\leq p_1$. By Proposition \ref{prop:elementaryembeddingBTL} and $q_0\leq p_0$ the embedding
$B^{s_0}_{p_0,q_0}(\R^d,w_0;X) \hookrightarrow F^{s_0}_{p_0,p_0}(\R^d,w_0;X)$ holds. Therefore, Proposition \ref{prop:intoweightedLpB} follows from Proposition \ref{prop:intoweightedLpF} in this case. The embeddings
\[F^{s_0}_{p_0,q_0}(\R^d,w_0;X)\hookrightarrow F^{0}_{p_1,1}(\R^d,w_1;X)\hookrightarrow L^{p_1}(\R^d,w_1;X)\]
are consequences of Theorem \ref{thm:main2} and Proposition \ref{prop:squeezeHF} (resp. Remark \ref{remarkProp3.12}).

If $p_0>p_1$ and $q_0=1$, then by Theorem \ref{thm:main1}
\[B^{s_0}_{p_0,1}(\R^d,w_0;X)\hookrightarrow B^{0}_{p_1,1}(\R^d,w_0;X)\hookrightarrow  L^{p_1}(\R^d,w_1;X),\]
where the last embedding follows again from Proposition \ref{prop:squeezeHF} (resp. Remark \ref{remarkProp3.12}).
\end{proof}

\begin{remark}
Many other results can be derived from Theorems \ref{thm:main1} and \ref{thm:main2}. Moreover,
Employing Proposition \ref{prop:elementaryembeddingBTL}, we see that a similar result as in Proposition \ref{prop:intoweightedLpF} above holds for $H^{s_0,p_0}(\R^d,w_0;X)$ and $W^{s_0,p_0}(\R^d,w_0;X)$ if $\gamma_0<d(p_0-1)$.
\end{remark}

For $m\in \N$, let $BU\!C^m(\R^d;X)$ denote the space of $m$-times differentiable functions with bounded and uniformly continuous derivatives. For $s=[s]+ s_*$ with $[s]\in \N_0$ and $s_*\in (0,1)$, let further $BU\!C^s(\R^d;X)$ the subspace of $BU\!C^{[s]}(\R^d;X)$ consisting of functions with $s_*$-H\"older continuous derivatives of order $[s]$.

\begin{proposition} \label{embedding-C}
Let $X$ be a Banach space, let $1<p_0<\infty$, $q_0\in [1, \infty]$ and $s_0\in \R$. Let $w_0(x) = |x|^{\gamma_0}$ with $\gamma_0\geq 0$.
If $s_1 = s_0-\frac{d+\gamma_0}{p_0}>0$ is not an integer, then
\begin{equation}\label{noninteger}
 E^{s_0,p_0}(\R^d,w_0;X)\hookrightarrow BU\!C^{s_1}(\R^d;X),
\end{equation}
where $E^{s_0,p_0}\in \{F^{s_0}_{p_0,q_0}, B^{s_0}_{p_0,q_0}\}$. If $m = s_0-\frac{d+\gamma_0}{p_0}\geq 0$ is an integer, then
\begin{equation}\label{integer}
B^{s_0}_{p_0,1}(\R^d,w_0;X)\hookrightarrow BU\!C^{m}(\R^d;X).
\end{equation}
Assuming that $0\leq \gamma_0 < d(p_0-1)$, these embeddings are also valid for $E^{s_0,p_0} = H^{s_0,p_0}$ and, if $s_0\in \N$, for $E^{s_0,p_0} = W^{s_0,p_0}$.
\end{proposition}
\begin{proof}
To prove (\ref{noninteger}), by the Propositions \ref{prop:elementaryembeddingBTL} and \ref{prop:squeezeHF} it suffices to show that
\[B^{s_0}_{p_0,\infty}(\R^d,w_0;X)\hookrightarrow BU\!C^{s_1}(\R^d;X).\]
This embedding is a consequence of Theorem \ref{thm:main1} and
\[B^{s_1}_{\infty, \infty}(\R^d,w_0;X) =B^{s_1}_{\infty, \infty}(\R^d;X) = BU\!C^{s_1}(\R^d;X),\]
where the latter identity can be proved as in the scalar case (see \cite[Theorem 2.5.7]{Tri83}).

For (\ref{integer}), Theorem \ref{thm:main1} yields
\[B^{s_0}_{p_0,1}(\R^d,w_0;X)\hookrightarrow B^{m}_{\infty,1}(\R^d;X).\]
We further have
\[B^{m}_{\infty,1}(\R^d;X) \hookrightarrow  BU\!C^{m}(\R^d;X),\]
due to Proposition \ref{prop:differentiation} and (the proof of) \cite[Theorem 2.5.7]{Tri83}.
\end{proof}

\def\polhk#1{\setbox0=\hbox{#1}{\ooalign{\hidewidth
  \lower1.5ex\hbox{`}\hidewidth\crcr\unhbox0}}} \def\cprime{$'$}
  \def\cprime{$'$}

\end{document}